\documentclass[a4paper,10pt]{amsart}

\usepackage[T1]{fontenc}

\usepackage[utf8]{inputenc}

\usepackage{lmodern}

\usepackage[english]{babel}

\usepackage{amsmath}

\usepackage{amssymb}

\usepackage{mathtools}

\usepackage{mathrsfs}

\usepackage{bookmark}

\usepackage{hyperref} 
\hypersetup{
	colorlinks=true,
	linkcolor=blue,
	filecolor=magenta,      
	urlcolor=cyan,
}

\usepackage[alphabetic]{amsrefs}

\usepackage{enumitem}
\usepackage{graphicx}
\usepackage{tikz}
\usepackage{tikz-cd}
\usetikzlibrary{matrix,arrows,decorations.pathmorphing}
\usepackage[all,cmtip]{xy}

\newtheorem*{thm-no-num}{Theorem}
\newtheorem*{df-no-num}{Definition}
\newtheorem{thm}{Theorem} [section]
\newtheorem{prop}[thm]{Proposition} 
\newtheorem{lm}[thm]{Lemma} 
\newtheorem{cor}[thm]{Corollary} 
\theoremstyle{remark}
\newtheorem{rmk}[thm]{Remark}

\theoremstyle{definition} 
\newtheorem {df}[thm]{Definition}

\newcommand{\bA}{\mathbb{A}}
\newcommand{\PP}{\mathbb{P}}
\newcommand{\ZZ}{\mathbb{Z}}
\newcommand{\OO}{\mathcal{O}}
\newcommand{\cl}[1]{\mathcal{#1}}
\newcommand*{\sheafhom}{\mathcal{H}\kern -.5pt om}

\newcommand{\Gm}{\mathbb{G}_m}
\newcommand{\GLt}{\textnormal{GL}_3}
\newcommand{\PGLt}{\textnormal{PGL}_2}
\newcommand{\SLt}{\textnormal{SL}_2}

\newcommand{\ccan}{\omega_{C/S}}

\newcommand{\sm}{_{\rm sm}}
\newcommand{\pr}{{\rm pr}}

\begin{document}
\title[The Chow ring of the stack of hyperelliptic curves of odd genus]{The Chow ring of the stack of hyperelliptic curves of odd genus}
\author[A. Di Lorenzo]{Andrea Di Lorenzo}
\address{Scuola Normale Superiore, Piazza Cavalieri 7, 56126 Pisa, Italy}
\email{andrea.dilorenzo@sns.it}
\date{\today}

\begin{abstract}
	We find a new presentation of the stack of hyperelliptic curves of odd genus as a quotient stack and we use it to compute its integral Chow ring by means of equivariant intersection theory.
\end{abstract}
\maketitle
\tableofcontents
\section*{Introduction}
\noindent
There is a well defined intersection theory with integral coefficients for quotient stacks, first developed in \cite{EdiGra}, generalizing some ideas contained in \cite{Tot}. In \cite{EdiGra} the authors defined the integral Chow ring $A^*(\cl{X})$ of a smooth quotient stack $\cl{X}=[U/G]$, and they also showed that if $\cl{X}$ is Deligne-Mumford, then the ring $A^*(\cl{X})\otimes\mathbb{Q}$ coincides with the rational Chow ring of Deligne-Mumford stacks, whose notion had already been introduced in \cite{Gil,Mum,Vis89}.

Since then, some explicit computations of integral Chow rings of interesting algebraic stacks have been carried on: in \cite{EdiGra} the authors computed $A^*(\overline{\cl{M}}_{1,1})$, the integral Chow ring of the compactified moduli stack of elliptic curves, and Vistoli in the appendix \cite{Vis} computed $A^*(\cl{M}_2)$, the integral Chow ring of the moduli stack of curves of genus $2$.
Furthermore, among moduli stack of curves, the integral Chow ring of the stack of at most $1$-nodal rational curves had been computed in \cite{EdiFul08} and the integral Chow ring of the stack of hyperelliptic curves of even genus had been determined explicitly in \cite{EdiFul}.

The main goal of this paper is to compute the integral Chow ring of $\cl{H}_g$, the moduli stack of hyperelliptic curves of genus $g$, when $g\geq 3$ is an odd number. Our main result is the following:
\begin{thm-no-num}
	$A^*(\cl{H}_g)=\mathbb{Z}[\tau,c_2,c_3]/(4(2g+1)\tau,8\tau^2-2g(g+1)c_2,2c_3) $
\end{thm-no-num}
\noindent
We also provide a geometrical interpretation of the generators of this ring.

The content of the theorem above had already been presented in the paper \cite{FulViv}, but recently R.Pirisi pointed out a mistake in the proof of \cite{FulViv}*{lemma $5.6$} which is crucial in order to complete the computation (for a more detailed analysis of this, see lemma \ref{lm:2-divisibility} and the following remark). Actually, the content of corollary \ref{cor:generators im i} implies that the proof of \cite{FulViv}*{lemma $5.6$} cannot be fixed, because its consequences, in particular \cite{FulViv}*{lemma $5.3$}, are wrong.

The main difference between our methods and the ones in \cite{FulViv} consists of the presentation of $\cl{H}_g$ as a quotient stack that is used in order to carry on the equivariant computations. Indeed, in \cite{FulViv} the authors exploit a presentation that had been first obtained in \cite{ArsVis}, which involves the algebraic group $\PGLt\times\Gm$. 
The group $\PGLt$ is a non-special group, i.e. there exist $\PGLt$-torsors over certain base schemes that are not Zariski-locally trivial but only \'{e}tale-locally trivial.
A consequence of this fact, which can be interpreted in numerous distinct ways, is that in general equivariant computations involving $\PGLt$ may be hard to carry on. For instance, not every projective space endowed with an action of $\PGLt$ can be seen as the projectivization of a $\PGLt$-representation: this makes the computation of the $\PGLt$-equivariant Chow ring of $\PP^1$ a non trivial challenge.

On the other hand, equivariant computations involving a special group, i.e. a group $G$ such that every $G$-torsor can be trivialized Zariski-locally, are more approachable. An important example of special group is the general linear group ${\rm GL}_n$. A key result of the present work is the following theorem, where a presentation of $\cl{H}_g$ as a quotient stack with respect to the action of a special group is explicitly obtained:
\begin{thm-no-num}
	There exists a scheme $U'$ such that $\cl{H}_g=[U'/\GLt\times\Gm]$
\end{thm-no-num}

To obtain the presentation above, we introduce the notions of $\GLt$-counterpart of a $\PGLt$-scheme and of $\GLt\times\Gm$-counterpart of a $\PGLt\times\Gm$-scheme. More precisely, we give the following definitions:
\begin{df-no-num}
	Let $k$ be a field, and let $X$ be a scheme of finite type over ${\rm Spec}(k)$ endowed with a $\PGLt$-action. Then the $\GLt$-counterpart of $X$ is a scheme $Y$ endowed with a $\GLt$-action such that $[Y/\GLt]\simeq[X/\PGLt]$.
\end{df-no-num}
\begin{df-no-num}
	Let $k$ be a field, and let $X$ be a scheme of finite type over ${\rm Spec}(k)$ endowed with a $\PGLt\times\Gm$-action. Then the $\GLt\times\Gm$-counterpart of $X$ is a scheme $Y$ endowed with a $\GLt\times\Gm$-action such that $[Y/\GLt\times\Gm]\simeq[X/\PGLt\times\Gm]$. 
\end{df-no-num}
We show that every $\PGLt$-scheme (resp. $\PGLt\times\Gm$-scheme) admits a $\GLt$-counterpart (resp. $\GLt\times\Gm$-counterpart), by explicit construction. These results are then applied to produce a new description of $\cl{H}_g$ as a quotient stack: indeed, the presentation contained in \cite{ArsVis} is of the form
$$ \cl{H}_g\simeq [(\bA(1,2g+2)\setminus\Delta')/\PGLt\times\Gm] $$
where $\bA(1,2g+2)$ denotes the affine space of binary forms of degree $2g+2$, the closed subscheme $\Delta'$ is the hypersurface parametrising binary forms with multiple roots and the action is defined as
$$(A,\lambda)\cdot (f(x,y))=\lambda^{-2}\det(A)^{g+1}f(A^{-1}(x,y))$$
By describing explicitly the $\GLt\times\Gm$-counterpart of $\bA(1,2g+2)\setminus\Delta'$, we obtain the new presentation.

The fact that $\GLt\times \Gm$ is special enables us to use a set of new tools that were not available using the presentation of \cite{ArsVis}, and these new tools allows us to complete the computation of the integral Chow ring of the moduli stack of hyperelliptic curves of odd genus. 
\subsection*{Description of contents}
The whole paper is ideally divided in two parts, the first one that goes from section \ref{sec:prelim} to section \ref{sec:new pres} and is more stack-theoretical, the second one that goes from section \ref{sec:proj bundles} to the end which is more computational.

In section \ref{sec:prelim} we introduce the notion of $\GLt$-counterpart of a $\PGLt$-scheme, we prove the existence of a $\GLt$-counterpart for every $\PGLt$-scheme and some related results on equivariant Chow groups.
In section \ref{sec:new pres} we give a new presentation of $\cl{H}_g$, for $g\geq 3$ an odd number. We also introduce a class of vector bundles, denoted $V_n$, which will play a central role in the remainder of the notes.
In section \ref{sec:proj bundles}, we study some intersection theoretical properties of the projective bundles $\PP(V_n)$.
In section \ref{sec:chow 1} we begin the computation of the Chow ring of $\cl{H}_g$, obtaining the generators and some relations. 
Section \ref{sec:chow 2} is the technical core of the paper: here is where the new presentation of $\cl{H}_g$ obtained in the second section will prove to be particularly useful in order to find other relations for the Chow ring of $\cl{H}_g$.
The computation of $A^*(\cl{H}_g)$ is completed in section \ref{sec:chow 3}, where we also provide a geometrical interpretation of the generating cycles of $A^*(\cl{H}_g)$. For the convenience of the reader, a more detailed description of the contents can be found at the beginning of every section.

We assume the knowledge of the basic tools of equivariant intersection theory. An excellent survey of the techniques used in this paper may be found in \cite{FulVis}*{section $2$}. We adopt a notation which is slightly different from the one adopted in \cite{FulVis}: for us the ring $A^*_{\Gm^{\oplus n}}({\rm Spec}(k))$ is generated by the cycles $\lambda_i$, for $i=1,...,n$, and the Chern classes $c_i$ are the elementary symmetric polynomials in $\lambda_1,...,\lambda_n$ of degree $i$. In particular, this means that the $\GLt$-equivariant ring of a point is generated by $c_1$, $c_2$ and $c_3$. We will use this notation throughout the present work.
\subsection*{Acknowledgements}I wish to thank my advisor Angelo Vistoli for the huge amount of precious time spent talking with me on this and related subjects, and for his uncommon generosity. I also wish to thank Roberto Pirisi for his carefully reading of a preliminary version of this paper.
\section{Preliminaries on $\PGLt$-schemes}\label{sec:prelim}
Fix a base field $k$. We begin with the following definitions.
\begin{df}\label{def: GLt-counterpart of a scheme}
	Let $X$ be a scheme of finite type over ${\rm Spec}(k)$ endowed with a $\PGLt$-action. Then a $\GLt$-counterpart of $X$ is a scheme $Y$ endowed with a $\GLt$-action and an isomorphism $\varphi:[Y/\GLt]\simeq[X/\PGLt]$.
\end{df} 
\begin{df}\label{def:GLt counterpart of a morphism}
	Let $X$ and $X'$ be two schemes of finite type over ${\rm Spec}(k)$ endowed with a $\PGLt$-action, and let $f:X\to X'$ be a $\PGLt$-equivariant morphism. Then a $\GLt$-counterpart of $f$ is the datum of: 
	\begin{enumerate}
		\item[(i)] A $\GLt$-equivariant morphism $g:Y\to Y'$ between two schemes endowed with a $\GLt$-action.
		\item[(ii)] Two isomorphisms of quotient stacks $\varphi:[Y/\GLt]\simeq [X/\PGLt]$ and $\varphi':[Y'/\GLt]\simeq [X'/\PGLt]$.
	\end{enumerate}		
	such that the diagram
	$$\xymatrix{
		[X/\PGLt]\ar[r] \ar[d]^{\varphi} & [X'/\PGLt]\ar[d]^{\varphi'} \\
		[Y/\GLt]\ar[r] & [Y'/\GLt] }$$
	commutes, where the top arrow (resp. the bottom arrow) is induced by $f$ (resp. by $g$). 
\end{df}
For sake of readability, when dealing with $\GLt$-counterparts we will only write the scheme or the morphism, without explicitly describing the isomorphism, which will be clear from our construction of $\GLt$-counterparts (see the proof of theorem \ref{prop:counterpart exists}).

The existence of a $\GLt$-counterpart of a $\PGLt$-scheme has some consequences on equivariant Chow groups. Recall from \cite{EdiGra} that if $X$ is a scheme of finite type over ${\rm Spec}(k)$ on which an algebraic group $G$ acts, we can form the equivariant Chow groups $A_i^G(X)$, which can be shown to only depend on the quotient stack $[X/G]$, and thus can be thought as the integral Chow groups of $[X/G]$.

Moreover, if $X\to X'$ is a proper $G$-equivariant morphism between two schemes both endowed with a $G$-action, there is an induced pushforward morphism between $A^G_i(X)$ and $A^G_i(X')$, that seen as a morphism between the Chow groups of the quotient stacks $A_i([X/G])$ and $A_i([X'/G])$ coincides with the pushforward morphism induced by the representable morphism $[X/G]\to [X'/G]$. 

From this we deduce the following result:
\begin{thm}\label{cor:chow diagrams}
	Let $f:X\to X'$ be a $\PGLt$-equivariant proper morphism between two schemes of finite type over ${\rm Spec}(k)$ both endowed with a $\PGLt$-action, and let $g:Y\to Y'$ be its $\GLt$-equivariant counterpart. Then there exists a commutative diagram of equivariant Chow groups of the form
	$$\xymatrix {
		A^{\PGLt}_i(X) \ar[r]^{f_*} \ar[d]^*[@]{\cong} & A^{\PGLt}_i(X') \ar[d]^*[@]{\cong} \\
		A^{\GLt}_i(Y) \ar[r]^{g_*} & A^{\GLt}_i(Y') }$$
\end{thm}
The following theorem assures us that definitions \ref{def: GLt-counterpart of a scheme} and \ref{def:GLt counterpart of a morphism} are not useless:
\begin{thm}\label{prop:counterpart exists}
	Let $f:X\to X'$ be a proper $\PGLt$-equivariant morphism between $\PGLt$-schemes. Then the morphism $f$ always admits a $\GLt$-counterpart.
\end{thm}
In particular, the theorem above tells us that given a $\PGLt$-scheme $X$, we can always find a $\GLt$-counterpart $Y$. The remainder of this section is devoted to the proof of this statement.

Recall the definition of the moduli stack $\cl{M}_0$ of smooth curves of genus $0$ whose objects are relative schemes $(C\to S)$, where $C\to S$ is a smooth and proper morphism whose fibres are curves of genus $0$. For us, a curve over a field is a geometrically connected, proper scheme of dimension $1$. From now on, the relative scheme $C\to S$ will be called a \textit{family of rational curves}.

It is well known that $\cl{M}_0$ is an algebraic stack isomorphic to the classifying stack $\cl{B}\PGLt$, and thus isomorphic to the quotient stack $[{\rm Spec}(k)/\PGLt]$, where the action of the group on the point is the trivial one. In other words, the point ${\rm Spec}(k)$ is a $\PGLt$-torsor over $\cl{M}_0$. 

The next proposition gives us another presentation of $\cl{M}_0$ as a quotient stack, but before its statement we need to introduce some additional notation:  let $\bA(2,2)$ be the affine space parametrising trinary forms of degree $2$, and let $\cl{S}\subset \bA(2,2)$ be the open subscheme of $\bA(2,2)$ parametrising smooth trinary forms of degree $2$. Then we have:
\begin{prop} \label{cor:A_is_torsor}
	The scheme $\cl{S}$ is a $\GLt$-torsor over $\cl{M}_0$ with action defined as $ A\cdot q(\underline{X}):=\det(A)q(A^{-1}\underline{X}) $, where $\underline{X}=(X_0,X_1,X_2)$. In particular, there is an isomorphism $[\cl{S}/\GLt]\simeq \cl{M}_0$.
\end{prop}
The proof of this proposition is postponed to the end of the section, as for the moment we prefer to show how it can be applied to prove theorem \ref{prop:counterpart exists}.
\begin{proof}[Proof of th. \ref{prop:counterpart exists}]
	Let $X$ be a $\PGLt$-scheme. Then we can form the quotient stack $[X/\PGLt]$, which fits into the cartesian diagram
	$$ \xymatrix{
		X \ar[r] \ar[d] & [X/\PGLt] \ar[d] \\
		{\rm Spec}(k) \ar[r] & \cl{M}_0 } $$
	It is easy to check that the morphism of stacks $[X/\PGLt]\to \cl{M}_0$ is representable. We can form another cartesian diagram of the form
	$$ \xymatrix{
		Y \ar[r] \ar[d] & [X/\PGLt] \ar[d] \\
		\cl{S} \ar[r] & \cl{M}_0 } $$
	and the representability of the right vertical morphism implies that $Y$ is a scheme. Moreover, we see that $Y$ must be a $\GLt$-torsor over the stack $[X/\PGLt]$. In other terms, there is an isomorphism $[Y/\GLt]\simeq [X/\PGLt]$.
	
	If $X\to X'$ is a proper $\PGLt$-equivariant morphism between two schemes both endowed with a $\PGLt$-action, this induces a proper morphism of quotient stacks $[X/\PGLt]\to [X'/\PGLt]$. We can actually pull back this proper morphism along $\cl{S}\to \cl{M}_0$, obtaining in this way a proper $\GLt$-equivariant morphism $Y\to Y'$, where $Y'$ is the $\GLt$-counterpart of $X'$. In particular, the induced morphism $[Y/\GLt]\to [Y'/\GLt]$, composed in the obvious way with the two isomorphisms $[X/\PGLt]\simeq [Y/\GLt]$ and $[Y'/\GLt]\simeq [X'/\PGLt]$, is equal to the morphism $[X/\PGLt]\to [X'/\PGLt]$.
\end{proof}
\begin{rmk}
	There is another way to obtain the $\GLt$-counterpart of a $\PGLt$-scheme. Recall that if we have a $G$-torsor $X\to Z$ and a morphism of algebraic groups $\varphi:G\to H$, we can construct the associated $H$-torsor of $X$ as $X\times^{G} H=(X\times H)/G$, where the (right) action of $G$ is
	$$ g\cdot (x,h)=(xg,\varphi(g)^{-1}h) $$
	Consider now the morphism of algebraic groups $\PGLt\to\GLt$ induced by the adjoint representation of $\PGLt$. Then this morphism permits us to produce from the $\PGLt$-torsor $X\to [X/\PGLt]$ a $\GLt$-torsor
	$$ Y=X\times^{\PGLt} \GLt \longrightarrow [X/\PGLt] $$
	and it can be checked that $Y$ is the $\GLt$-counterpart of $X$.
\end{rmk}
Now we give a proof of proposition \ref{cor:A_is_torsor} stated at the beginning. We start with two technical lemmas:
\begin{lm} \label{lm:triv}
	Let $L$ be an invertible sheaf on a scheme $\pi:X\rightarrow S$ such that $\pi_*L$ is a globally generated locally free sheaf of rank $n+1$. Then giving an isomorphism $\pi_*L\simeq\mathcal{O}_S^{\oplus n+1}$ induces a morphism $f:X\rightarrow\mathbb{P}_S^n$ and an isomorphism $f^*\mathcal{O}(1)\simeq L$, and vice versa.
\end{lm}
\begin{proof}
	One implication is obvious. For the other one, suppose to have a morphism $f:X\rightarrow\mathbb{P}_S^n$ and an isomorphism $f^*\mathcal{O}(1)\simeq L$. For functoriality of the pushforward we obtain an isomorphism $\pi_*L\simeq \pi_*f^*\OO(1)$. We want to produce an isomorphism of the sheaf on the right with $\OO_S^{\oplus n+1}$.
	
	Let ${\rm pr}_2:\PP^n\times S\to S$ be the canonical projection. Then we have a surjective morphism of locally free sheaves 
	$ {\rm pr}_2^*{\rm pr}_{2*}\OO(1) \to \OO(1) $
	that composed with $f^*$ and $\pi_*$ induces a surjective morphism of locally free sheaves
	$ \alpha:\pi_*\pi^* {\rm pr}_{2*}\OO(1) \to \pi_*f^*\OO(1) $.
	Observe that we have a canonical chain of isomorphisms: \[\OO_S^{\oplus n+1}\simeq \pi_*\pi^*\OO^{\oplus n+1}_S\simeq \pi_*\pi^*{\rm pr}_{2*}\OO(1)\] Thus, composing these isomorphisms with $\alpha$ we get a surjective morphism of locally free sheaves
	$\OO_S^{\oplus n+1} \to \pi_*f^*\OO(1) $. The hypotheses imply that this morphism is actually an isomorphism.
\end{proof}
\begin{lm} \label{lm:bc}
	Let $\pi:C\rightarrow S$ be a family of rational curves. Then the sheaf $\pi_*\ccan^{-1}$ is a locally free sheaf on $S$ of rank $3$ which satisfies the base change property. The morphism $\pi^*\pi_*\ccan^{-1}\rightarrow \ccan^{-1}$ is surjective and induces a closed immersion $C\hookrightarrow\mathbb{P}(\pi_*\ccan^{-1})$.
\end{lm}
\begin{proof}
	Follows from the base change theorem in cohomology (\cite{Har}*{Th. III.12.11}) applied to $\pi_*\ccan^{-1}$.
\end{proof}
Consider the prestack in groupoids over the category $({\rm Sch}/k)$ of schemes
$$\cl{E}(S)=\left\{ (\pi:C\to S,\alpha:\pi_*\omega_{C/S}^{-1}\simeq \OO_S^{\oplus 3}) \right\}$$
where $C\to S$ is a family of rational curves, and the morphisms $$(C\to S,\alpha:\pi_*\omega_{C/S}^{-1}\simeq \OO_S^{\oplus 3})\to (C'\to S',\alpha':\pi'_*\omega_{C'/S'}^{-1}\simeq \OO_{S'}^{\oplus 3})$$ are given by triples $(\varphi:S'\to S,\psi:C'\simeq \varphi^*C,\phi:\pi'_*\omega_{C'/S'}^{-1}\simeq\varphi^*\pi_*\omega_{C/S}^{-1})$, where $\phi$ must commute with $\alpha$ and $\alpha'$. It can be easily checked that this prestack is equivalent to a sheaf.

Observe that there is a free and transitive action of $\GLt$ on $\cl{E}$, which turns $\cl{E}$ into a $\GLt$-torsor sheaf over $\mathcal{M}_0$.
Consider also the auxiliary prestack 
$$\cl{E}'(S)=\left\{ (({\bf D}),\beta:i^*\OO(1)\simeq \omega_{C/S}^{-1})\right\}$$ 
where $({\bf D})$ is a commutative diagram of the form
$$ \xymatrix{
	C \ar@{^{(}->}[r]^i \ar[d] & \PP^2_S \ar[dl] \\
	S
}$$
with $C\to S$ a family of rational curves, and $i$ a closed immersion. Recall that $\cl{S}$ is the scheme parametrising smooth forms of degree $2$ in three variables.
\begin{lm}\label{lm:E'_equi_A}
	There are isomorphisms $\cl{E}\simeq\cl{E}'\simeq\cl{S}$.
\end{lm}
\begin{proof}
	The first isomorphism follows from lemma \ref{lm:triv}.
	Suppose to have a commutative triangle 
	$$ \xymatrix{
		C \ar@{^{(}->}[r]^i \ar[d] & \PP^2_S \ar[dl] \\
		S
	}$$
	and an isomorphism $\varphi:i^*\OO(1)\simeq \omega_{C/S}^{-1}$. This one can be seen as a non-zero section of $H^0(C,i^*\OO(1)\otimes\ccan)$. Now we have the following chain of isomorphisms:
	\begin{align*}
	H^0(C,i^*\OO(1)\otimes\ccan)&\simeq H^0(\PP^2_S,i_*(i^*\OO(1)\otimes\ccan))\\
	&\simeq H^0(\PP^2_S,i_*(i^*(\OO(1)\otimes\omega_{\PP^2_S}\otimes\cl{I}^{-1}))\\
	&\simeq H^0(\PP^2_S,\OO(-2)\otimes\cl{I}^{-1}\otimes i_*\OO_C)
	\end{align*}
	where $\cl{I}$ denotes the ideal sheaf of $i(C)\subset \PP^2_S$ and in the last line we used the projection formula and the canonical isomorphism $\omega_{\PP^2}\simeq \OO(-3)$. If $L:=\OO(-2)\otimes\cl{I}^{-1}$, then by twisting the exact sequence
	$$ 0\to \cl{I}\to \OO_{\PP^2_S}\to i_*\cl{\OO}_C\to 0 $$
	by $L$ and by taking the associated long exact sequence in cohomology, we easily deduce the isomorphism
	$$ H^0(\PP^2_S,L\otimes i_*\OO_C)\simeq H^0(\PP^2_S,L) $$
	Now observe that a non-zero global section of $L$ induces an isomorphism $\cl{I}\simeq\OO(-2)$, and vice versa. 
	
	Thus, by dualizing the injective morphism of sheaves $\cl{I}\hookrightarrow \OO_{\PP^2_S}$ and by applying the isomorphism above, we obtain a morphism $\OO_{\PP^2_S}\to\OO(2)$, which is equivalent to choosing a global section $q$ of $\OO(2)$, that will be smooth because of the hypotheses on $C$. 
	
	It is easy to check that the induced morphism
	$$\cl{E}'\longrightarrow \cl{S},\quad ((\boldsymbol{D}),\varphi)\longmapsto q$$	
	is an isomorphism, whose inverse is given by defining $C$ as $Q\subset \PP^2_S$, where $Q$ is the zero locus of $q$.
\end{proof}
%\begin{cor} \label{cor:A_is_torsor}
%	The scheme $\cl{S}$ is a $\GLt$-torsor over $\cl{M}_0$ with action defined as $ A\cdot q(\underline{X}):=\det(A)q(A^{-1}\underline{X}) $. In particular, there is an isomorphism $[\cl{S}/\GLt]\simeq \cl{M}_0$.
%\end{cor}
\begin{proof}[Proof of prop. \ref{cor:A_is_torsor}]
	From lemma \ref{lm:E'_equi_A} we readily deduce proposition \ref{cor:A_is_torsor}. We only have to check that the action of $\GLt$ on $\cl{S}$ is the correct one, but this immediately follows from the isomorphism $\cl{I}\simeq\omega_{\PP^2_S}(1)$ seen in the proof of lemma \ref{lm:E'_equi_A}.
\end{proof}
Now we give other two definitions that are useful for our purposes:
\begin{df}\label{df: GLtxGm counterpart of a scheme}
	Let $X$ be a scheme of finite type over ${\rm Spec}(k)$ endowed with a $\PGLt\times\Gm$-action. Then a $\GLt\times\Gm$-counterpart of $X$ is a scheme $Y$ endowed with a $\GLt\times\Gm$-action and an isomorphism $\varphi:[Y/\GLt\times\Gm]\simeq[X/\PGLt\times\Gm]$.
\end{df}
\begin{df}\label{def:GLtxGm counterpart of a morphism}
	Let $X$ and $X'$ be two schemes of finite type over ${\rm Spec}(k)$ endowed with a $\PGLt\times\Gm$-action, and let $f:X\to X'$ be a $\PGLt\times\Gm$-equivariant morphism. Then a $\GLt\times\Gm$-counterpart of $f$ is the datum of: 
	\begin{enumerate}
		\item[(i)] a $\GLt\times\Gm$-equivariant morphism $g:Y\to Y'$ between two schemes endowed with a $\GLt\times\Gm$-action,
		\item[(ii)] Two isomorphisms $\varphi:[Y/\GLt\times\Gm]\simeq [X/\PGLt\times\Gm]$ and $\varphi':[Y'/\GLt\times\Gm]\simeq [X'/\PGLt]$,
		such that the diagram
		$$\xymatrix{
			[X/\PGLt\times\Gm]\ar[r] \ar[d]^{\varphi} & [X'/\PGLt\times\Gm]\ar[d]^{\varphi'} \\
			[Y/\GLt\times\Gm]\ar[r] & [Y'/\GLt\times\Gm] }$$
	\end{enumerate}
	commutes, where the top arrow (resp. the bottom arrow) is induced by $f$ (resp. by $g$).
\end{df}
Then, just as in the previous case, we have:
\begin{thm}\label{prop:Gm counterpart exists}
	Let $f:X\to X'$ be a proper $\PGLt\times\Gm$-equivariant morphism between $\PGLt\times\Gm$-schemes. Then it always admits a $\GLt\times\Gm$-counterpart.
\end{thm}
\begin{proof}
	The proof of the theorem above works exactly in the same way as the proof of theorem \ref{prop:counterpart exists}: one has only to take into account the action of $\Gm$, but is immediate to check that $\cl{S}$ is a $\GLt\times\Gm$-torsor over $\cl{B}(\PGLt\times\Gm)$, where the action of $\Gm$ is the trivial one. From this the theorem follows.
\end{proof}
\section{A new presentation of $\cl{H}_g$ as a quotient stack}\label{sec:new pres}
\noindent
Fix a base field $k$ of characteristic different from $2$ and an odd integer $g\geq 3$. Let us stress the fact that $g$ will always be odd, as this is a key property in most of the constructions presented here. Recall that by a \textit{family of rational curves} over $S$ we mean a proper and smooth scheme over a $k$-base scheme $S$ such that every fiber is a connected curve of genus $0$. Then a \textit{family of hyperelliptic curves} of genus $g$ over $S$ is defined as a pair $(C\to S,\iota)$ where $C\to S$ is a proper and smooth scheme over a base $k$-scheme $S$ such that every fiber is a connected curve of genus $g$, and $\iota\in{\rm Aut}(C)$ is an involution such that $C/\langle\iota\rangle\to S$ is a family of rational curves.

Let $\cl{H}_g$ be the moduli stack of smooth hyperelliptic curves of genus $g$, whose objects are families of hyperelliptic curves of genus $g$ as defined above, and the morphisms are the isomorphisms over $S$ (the condition of commuting with the involutions is automatically satisfied).
The goal of this section is to give a presentation of this stack as a quotient stack $[U'/\GLt\times\Gm]$, where $U'$ is an certain scheme that will be defined later. This is done in theorem \ref{thm:presentation_Hg}.
\subsection{Properties of hyperelliptic curves}
We briefly recall some basic facts about hyperelliptic curves (for an extensive treament see \cite{KnuKle}). Let $(C\to S,\iota)$ be a family of hyperelliptic curves of genus $g$. By definition $\iota$ induces the hyperelliptic involution on every geometric fiber, hence an $S$-morphism $f:C\to C'$ of degree $2$ which on each geometric fiber corresponds to taking the quotient w.r.t. the hyperelliptic involution. The scheme $C'\to S$ is a smooth family of rational curves.

Families of hyperelliptic curves have a canonical subscheme $W_{C/S}$, called the \textit{Weierstrass subscheme}, that is the ramification divisor of $f$ endowed with the scheme structure given by the zeroth Fitting ideal of $\Omega^1_{C/C'}$. It is finite and \'{e}tale over $S$ of degree $2g+2$, and its associated line bundle, when seen as an effective Cartier divisor, is the dualizing sheaf $\omega_f$ relative to the finite morphism $f$. Clearly, $f$ induces an isomorphism between $W_{C/S}$ and the branch divisor $D$ on $C'$. 
\subsection{Preliminaries on $\cl{H}_g$}
Let $\bA(1,2g+2)\sm$ be the scheme parametrising smooth binary forms of degree $2g+2$. There is an action of $\PGLt\times\Gm$ over this scheme defined as follows: $$(A,\lambda)\cdot (f(x,y))=\lambda^{-2}\det(A)^{g+1}f(A^{-1}(x,y))$$
Observe that the above action is well defined, though the determinant of an element of $\PGLt$ is not.
In \cite{ArsVis} the authors proved that $\cl{H}_g$ is isomorphic to the quotient stack $[\bA(1,2g+2) \sm/\PGLt\times\Gm]$.

Therefore, a new presentation of $\cl{H}_g$ as a quotient stack with respect to the action of $\GLt\times\Gm$ can be obtained by finding a $\GLt\times\Gm$-counterpart of the $\PGLt\times\Gm$-scheme $\bA(1,2g+2)\sm$.

Actually, in this section we will also study some $\GLt$-counterparts and $\GLt\times\Gm$-counterparts of other schemes that will be relevant for our purposes.
\subsection{Computation of $\GLt$-counterparts}
Let $\bA(1,2n)$ be the affine space of the homogeneous polynomials of degree $2n$ in two variables. There is an action of $\PGLt$ on this scheme given by:
$$A\cdot f(x,y):=\det(A)^nf(A^{-1}(x,y))$$
This induces an action of $\PGLt$ on $\PP(1,2n)$. We want to find a $\GLt$-counterpart of this $\PGLt$-scheme.

As in section \ref{sec:prelim}, let $\bA(2,2)$ be the affine space of trinary forms of degree $2$, and let $\PP(2,2)$ be its projectivization. Call $\cl{Q}$ the universal conic over $\PP(2,2)$ and denote $\pr_i$ the projection of the scheme $\PP(2,2)\times \PP^2$ onto its $i^{\rm th}$-factor. Then we have the following exact sequence:
$$ 0\longrightarrow \cl{I}_{\cl{Q}} \longrightarrow \OO_{\PP(2,2)\times \PP^2}\longrightarrow i_*\OO_{\cl{Q}} \longrightarrow 0 $$
After tensoring with $\pr_2^*\OO_{\PP^2}(n)$ and pushing forward along $\pr_1$, due to the fact that $R^1\pr_{1*}(\cl{I}_{\cl{Q}}\otimes\pr_2^*\OO_{\PP^2}(n))$ vanishes, we obtain the short exact sequence:
$$ 0\longrightarrow \pr_{1*}(\cl{I}_{\cl{Q}}\otimes\pr_2^*\OO_{\PP^2}(n)) \longrightarrow \pr_{1*}\pr_2^*\OO_{\PP^2}(n) \longrightarrow \pr_{1*}(i_*\OO_{\cl{Q}}\otimes\pr_2^*\OO_{\PP^2}(n)) \longrightarrow 0$$
The sheaf in the center is isomorphic to $H^0(\PP^2,\OO(n))\otimes \OO_{\PP(2,2)\times \PP^2}$ and hence it is free. Observe also that, due to the isomorphism $\cl{I}_{\cl{Q}}\simeq \pr_1^*\OO(-1)\otimes \pr_2^*\OO_{\PP^2}(-2)$, the sheaf on the left is isomorphic to the locally free sheaf $H^0(\PP^2,\OO(n-2))\otimes \OO(-1)$. 

This implies that the last morphism is surjective and that the sheaf on the right of the sequence is locally free.
\begin{df} \label{def:Vn}
	\hspace{1pt}
	\begin{enumerate}
		\item We define $\overline{V}_n$ as the vector bundle $\pr_{1*}(i_*\OO_{\cl{Q}}\otimes\pr_2^*\OO_{\PP^2}(n))$ on $\PP(2,2)$.
		\item We define $V'_n$ as the vector bundle over $\bA(2,2)\setminus\{0\}$ obtained by pulling back $\overline{V}_n$ along the $\Gm$-torsor $\bA(2,2)\setminus\{0\} \to \PP(2,2)$.
		\item We define the vector bundle $V_n$ as the restriction of $V'_n$ to $\cl{S}$, the open subscheme of $\bA(2,2)$ parametrising quadratic ternary forms of rank $3$.
	\end{enumerate}
\end{df}
\begin{rmk}\label{rmk:point of Vn}
	Another description of $V'_n$ is the following: let $\bA(n,d)$ be the affine space parametrising homogeneous polynomials (forms) of degree $d$ in $n+1$ variables.
	For $n\geq 2$ we can define an injective morphism of trivial vector bundles over $\bA(2,2)\setminus\{0\}$ as follows:
	$$ \bA(2,2)\setminus\{0\}\times\bA(2,n-2)\longrightarrow\bA(2,2)\setminus\{0\}\times\bA(2,n),\quad (q,f)\longmapsto (q,qf)$$
	The quotient is a locally free sheaf on $\bA(2,2)\setminus\{0\}$ and it coincides with $V_n'$.
	
	For any field $K$, the $K$-points of $V'_n$ can be thought as pairs $(q,f)$, where $q$ is quadratic trinary form of rank $3$ over $K$, and $f$ is an equivalence class of ternary forms of degree $n$ over $K$, where two such forms are equivalent if their difference is divisible by $q$.
\end{rmk}
\begin{prop}\label{prop:PV/GLt_iso_Hilb}
	The projective bundle $\PP(V_n)$ is a $\GLt$-counterpart of $\PP(1,2n)$.
\end{prop}
\begin{proof}
	The quotient $[\PP(1,2n)/\PGLt]$ is the stack whose objects are pairs $(C\to S,D)$, where $C\to S$ is a family of rational curves and $D\subset C$ is a divisor such that the induced morphism $D \to S$ is flat and finite of degree $2n$.
	
	Therefore the $\GLt$-counterpart of $\PP(1,2n)$ is the fibred product:
	$$[\PP(1,2n)/\PGLt]\times_{\cl{M}_0}\cl{S}$$ 
	Recall from lemma \ref{lm:E'_equi_A} that we have an isomorphism of $\cl{S}$, the scheme parametrising quadratic ternary forms of rank $3$, with $\cl{E'}$, the stack in sets whose objects are pairs $(({\bf D}),\beta)$, where $({\bf D})$ is a commutative diagram of the form
	$$ \xymatrix{
		C \ar@{^{(}->}[r]^i \ar[d] & \PP^2_S \ar[dl] \\
		S
	}$$
	with $C\to S$ a family of rational curves, $i$ a closed immersion and $\beta$ is an isomorphism of locally free sheaves $i^*\OO(1)\simeq \omega_{C/S}^{-1}$.
	
	Therefore the fibred product above can be described as the stack whose objects are triples $(({\bf D}),\beta,D)$, where $(({\bf D}),\beta)$ is an object of $\cl{E}'$ and $D\subset C$ is as before.
	
	First we construct a morphism from this stack to $\PP(V_n)$. Clearly, there is a natural morphism $p:[\PP(1,2n)/\PGLt]\times_{\cl{M}_0}\cl{S}\to \cl{S}$. 
	
	Given an object $(({\bf D}),\beta,D)$, consider the injective morphism of sheaves $\cl{I}_D\hookrightarrow \OO_C$. After twisting by $\OO(n)$ and pushing forward along $\pi:C\to S$, we get the injective morphism:
	$$ \pi_*\cl{I}_D(n)\longrightarrow \pi_*\OO_C(n) $$
	By cohomology and base change (see \cite{Har}*{Th. III.12.11}) we see that the first sheaf is invertible, whilst the second one, using the projection formula for vector bundles, can easily be proved to be isomorphic to $p^*V_n$. The well known characterization of the morphisms to projective bundles yields a morphism of sets $$[\PP(1,2n)/\PGLt]\times_{\cl{M}_0}\cl{S}(S)\to \PP(V_n)(S)$$
	As everything is functorial, we get a well defined morphism from the $\GLt$-counterpart of $\PP(1,2n)$ to $\PP(V_n)$.
	
	To construct an inverse of this morphism, consider the universal conic $\pi:\cl{Q}\to \PP(2,2)$ and pull it back to $\PP(\overline{V}_n)$ (see definition \ref{def:Vn}), so that we obtain a cartesian diagram:
	\[\xymatrix{
		\cl{P} \ar[r] \ar[d]^{\rho} & \cl{Q} \ar[d]^{\pi} \\
		\PP(\overline{V}_n) \ar[r] & \PP(2,2)}\]	
	By construction $\cl{P}$ is a closed subscheme of $\PP(\overline{V}_n)\times\PP^2$.
	The Euler exact sequence for $\PP(\overline{V}_n)$, pulled back to $\cl{P}$ and after some manipulations, yields an injective morphism:
	$$ \rho^*\OO_{\PP(\overline{V}_n)}(-1) \longrightarrow \rho^*\rho_*(\pr_2^*\OO_{\PP^2}(n)|_{\cl{P}}) $$
	Observe that there is a surjective morphism from the sheaf on the right to $\pr_2^*\OO_{\PP^2}(n)|_{\cl{P}}$. Hence, after twisting by $\pr_2^*\OO_{\PP^2}(-n)|_{\cl{P}}$ we can construct a morphism:
	$$ \rho^*\OO_{\PP(\overline{V}_n)}(-1)\otimes\pr_2^*\OO_{\PP^2}(-n)|_{\cl{P}} \longrightarrow \OO_{\cl{P}} $$
	It is immediate to verify that this defines a Cartier divisor $\cl{D}\subset\cl{P}$ whose associated sheaf of ideals is $\rho^*\OO_{\PP(\overline{V}_n)}(-1)\otimes\pr_2^*\OO_{\PP^2}(-n)|_{\cl{P}}$ and whose projection onto $\PP(\overline{V}_n)$ is flat and finite of degree $2n$.
	
	Pulling back the universal conic and the Cartier divisor to $\PP(V_n)$, we get an object of $[\PP(1,2n)/\PGLt]\times_{\cl{M}_0}\cl{S}(\PP(V_n))$. As everything is functorial, this defines a morphism $\PP(V_n) \to [\PP(1,2n)/\PGLt]\times_{\cl{M}_0}\cl{S}$. It is easy to check that this second morphism is the inverse of the first morphism that we defined, which concludes the proof of the proposition.
\end{proof}
\begin{rmk}
	The scheme $\PP(1,2n)$ can be thought as the Hilbert scheme ${\rm Hilb}^{2n}_{\PP^1_k}$ of $2n$ points on $\PP^1$. Its quotient $[\PP(1,2n)/\PGLt]$ can be identified with the Hilbert stack ${\rm Hilb}^{2n}_{P/\cl{B}\PGLt}$ of $2n$ points relative to the universal torsor ${\rm Spec}(k)$ over the classifying stack $\cl{B}\PGLt$. Equivalently, we can think of this stack as the Hilbert stack ${\rm Hilb}^{2n}_{\cl{C}/\cl{M}_0}$ of $2n$ points relative to the universal rational curve $\cl{C}$ over $\cl{M}_0$.
	
	Proposition \ref{prop:PV/GLt_iso_Hilb} gives us the following presentation of this stack as a quotient stack:
	$$ {\rm Hilb}^{2n}_{\cl{C}/\cl{M}_0}\simeq [\PP(V_n)/\GLt] $$
	Observe that the projective bundle $\PP(V_n)$ itself can be thought as the Hilbert scheme ${\rm Hilb}^{2n}_{\cl{Q}/\cl{S}}$ of $2n$ points relative to the univeral quadric $\cl{Q}$ over $\cl{S}$, the scheme parametrising quadratic ternary forms of rank $3$.
	
	An interesting feature of this new presentation is that provides us with a natural way to partially extend the Hilbert stack ${\rm Hilb}^{2n}_{\cl{C}/\cl{M}_0}$, which is a stack over $\cl{M}_0$, to a stack over the stack of genus $0$ and at most $1$ nodal curves $\cl{M}_0^{\leq 1}$. Indeed, instead of taking $\PP(V_n)$ we can take the projectivization of the vector bundle $V_n^{\leq 1}$ defined over $\bA(2,2)^{\leq 1}$, the scheme parametrising quadrics in three variables of rank strictly greater than $1$. Then the quotient stack $[\PP(V_n^{\leq 1})/\GLt]$ gives a natural enlargement of the Hilbert stack ${\rm Hilb}^{2n}_{\cl{C}/\cl{M}_0}$ to a stack over $\cl{M}_0^{\leq 1}$.
\end{rmk}
\begin{prop}\label{prop:GLt-counterpart of A(1,2n)}
	The vector bundle $V_n$ is a $\GLt$-counterpart of $\bA(1,2n)$.
\end{prop}
\begin{proof}
	Observe that $\bA(1,2n)\setminus \{0\}$ is the total space of the $\Gm$-torsor associated to the tautological line bundle $\OO_{\PP(1,2n)}(-1)$ of $\PP(1,2n)$. We can check, using proposition \ref{prop:PV/GLt_iso_Hilb}, that a $\GLt$-counterpart of this line bundle is given by $\OO_{\PP(V_n)}(-1)$. From this we deduce the the $\Gm$-torsor associated to this line bundle is a $\GLt$-counterpart of $\bA(1,2n)\setminus \{0\}$: such a torsor is precisely $V_n$.
\end{proof}
\begin{prop}\label{prop:GLtxGm-counterpart of A(1,2n)}
	Let $\Gm$ acts on $\bA(1,2n)$ by multiplication for $\lambda$. Then a $\GLt\times\Gm$-counterpart of $\bA(1,2n)$ is $V_n$, where $\Gm$ acts as
	$\lambda\cdot (q,f):=(q,\lambda f)$.
\end{prop}
\begin{proof}
	It follows from proposition \ref{prop:GLt-counterpart of A(1,2n)}.
\end{proof}
Let $\Delta'\subset \bA(1,2n)$ be the $\PGLt$-invariant, closed subscheme parametrising singular binary forms of degree $2n$. In other terms, the points of $\Delta'$ correspond to global sections $\sigma$ of $\OO_{\PP^1_k}(2n)$ with multiple roots. We want to find its $\GLt$-counterpart.

Consider the set $D'$ inside $V_{n}$ defined as follows:
$$D':=\left\{ (q,f)\text{ such that }V_+(q,f)\subset\PP^2\text{ is singular } \right\} $$
where $V_+(q,f)$ is the closed subscheme of $\PP^2$ defined by the homogeneous ideal $I=(q,f)$.
Let us show how to put a scheme structure on this set: consider the closed subscheme of $\cl{S}\times\bA(2,n)\times\PP^2$ defined as
$$D'':=\left\{ (q,f,u)\text{ such that }u\text{ is a singular point of }V_+(q,f) \right\}$$
Then $D''$ is a scheme, because it can be defined as the locus where
$$q(u)=f(u)=0,\quad \text{rk}(J(q,f)(u))\text{ is not maximal }$$
Here $J(q,f)$ is the Jacobian matrix of $q$ and $f$. Then the image of $D'$ via the proper morphism
$${\rm pr}: \cl{S}\times\bA(2,n)\times\PP^2 \longrightarrow \cl{S}\times\bA(2,n)$$
inherits a scheme structure, and projecting again ${\rm pr}(D'')$ along the quotient morphism
$$ \cl{S}\times\bA(2,n)\longrightarrow V_n$$
we obtain exactly $D'$, which in this way inherits a scheme structure.

Using the same arguments we used to prove proposition \ref{prop:GLt-counterpart of A(1,2n)}, we can deduce the following result:
\begin{prop}\label{prop:counterpart of delta'}
	We have:
	\begin{enumerate}
		\item $D'$ is a $\GLt$-counterpart and a $\GLt\times\Gm$-counterpart of $\Delta'$.
		\item $V_n\setminus  D'$ is a $\GLt$-counterpart and the $\GLt\times\Gm$-counterpart of $\bA(1,2n)\setminus\Delta'$.
	\end{enumerate}
\end{prop}
Let $\Delta\subset\PP(1,2n)$ denotes the image of $\Delta'$ via the projection $(\bA(1,2n)\setminus\{0\})\to\PP(1,2n)$, and let $D\subset\PP(V_n)$ be the projection of $D'$ along $(V_n\setminus\sigma_0)\to\PP(V_n)$.
\begin{cor}\label{cor:counterpart of delta}
	We have:
	\begin{enumerate}
		\item $D$ is a $\GLt$-counterpart of $\Delta$.
		\item $V_n\setminus  D$ is a $\GLt$-counterpart of $\PP(1,2n)\setminus\Delta$ .
	\end{enumerate}
\end{cor}
\subsection{Some comparison results}\label{subsec:comp res}
In the previous subsection, we found that $\PP(V_n)$ is a $\GLt$-counterpart (see definition \ref{def: GLt-counterpart of a scheme}) of the $\PGLt$-scheme $\PP(1,2n)$. We want now to apply theorem \ref{cor:chow diagrams} to this particular case.
There are two relevant classes of morphisms that we want to consider, which are
\begin{align*}
\psi_{n}:\PP(1,2n)\longmapsto\PP(1,4n),&\quad f\longmapsto f^2 \\
\psi_{n,m}:\PP(1,2n)\times\PP(1,2m)\longrightarrow\PP(1,2n+2m),&\quad (f,g)\longmapsto fg 
\end{align*}
All these maps are $\PGLt$-equivariant and it is immediate to verify that $\GLt$-counterparts of these morphisms (see definition \ref{def:GLt counterpart of a morphism}) are
\begin{align*}
\psi'_n:\PP(V_n)\longmapsto \PP(V_{2n}),&\quad (q,f)\longmapsto (q,f^2)\\
\psi'_{n,m}:\PP(V_n)\times_{\cl{S}}\PP(V_m)\longrightarrow \PP(V_{n+m}),&\quad (q,f,g)\longmapsto(q,fg)
\end{align*}
All the morphisms involved are proper, and theorem \ref{cor:chow diagrams} gives us the following commutative diagrams of equivariant Chow rings:
$$\xymatrix{
	A^*_{\GLt}(\PP(V_n)) \ar[r]^{\psi'_{n*}} \ar[d]^*[@]{\cong} & A^*_{\GLt}(\PP(V_{2n})) \ar[d]^*[@]{\cong} \\
	A^*_{\PGLt}(\PP(1,2n)) \ar[r]^{\psi_{n*}} & A^*_{\PGLt}(\PP(1,4n)) }$$
$$\xymatrix{
	A^*_{\GLt}(\PP(V_n))\otimes_{A^*_{\GLt}(\cl{S})}A^*_{\GLt}(\PP(V_m)) \ar[r]^{\hspace{1cm}\psi'_{n,m*}} \ar[d]^*[@]{\cong} & A^*_{\GLt}(\PP(V_{n+m})) \ar[d]^*[@]{\cong} \\
	A^*_{\PGLt}(\PP(1,2n))\otimes_{A^*_{\PGLt}} A^*_{\PGLt}(\PP(1,2m)) \ar[r]^{\hspace{1.2cm}\psi_{n,m*}} & A^*_{\PGLt}(\PP(1,2n+2m)) }$$
Every morphism obtained composing $\psi_{n,m}$ and $\psi_{n}$ induces a commutative diagram as the ones above. This will be one of the key tools used to compute the Chow ring of $\cl{H}_g$.
\subsection{The main result}
We are ready to give a new presentation of the stack $\cl{H}_g$ as a quotient stack.
\begin{thm}
	\label{thm:presentation_Hg}
	Let $U':=V_{g+1}\setminus D'$. Then we have an isomorphism $$\cl{H}_g\simeq [U'/\GLt\times\Gm]$$ where the action on $U'$ is given by the formula $$(A,\lambda)\cdot(q,f)=(\det(A)q(A^{-1}\underline{X}),\lambda^{-2}f(A^{-1}\underline{X}))$$ 
\end{thm}
\begin{proof}
	Recall the presentation of \cite{ArsVis}:
	$$ \cl{H}_g\simeq [\bA(1,2g+2)\setminus\Delta'/\PGLt\times\Gm]$$
	with action defined as
	$$ (A,\lambda)\cdot f(x,y):=\lambda^{-2}\det(A)^{g+1}f(A^{-1}(x,y)) $$
	Proposition \ref{prop:counterpart of delta'}.(2) tells us that a $\GLt\times\Gm$-counterpart of $\bA(1,2g+2)\setminus\Delta'$, with $\Gm$ acting by simple multiplication, is $V_{g+1}\setminus D'$, where the action of $\GLt\times\Gm$ is:
	$$(A,\lambda)\cdot (q,f):=(\det(A)q(A^{-1}(x,y,z),\lambda f(A^{-1}(x,y,z))$$
	It is easy to see that the $\GLt\times\Gm$-counterpart of $\bA(1,2g+2)\setminus\Delta'$ with $\Gm$ acting by $\lambda^{-2}$ is $V_{g+1}\setminus D'$, with $\Gm$ acting by multiplication for $\lambda^{-2}$.
\end{proof}
The theorem above can be rephrased by saying that $V_{g+1}\setminus D'$ is a $\GLt\times\Gm$-torsor over $\cl{H}_g$. It is well known that to every $\GLt\times\Gm$-torsor one can associate a rank $4$ vector bundle of the form $\cl{E}\oplus \cl{L}$, where $\cl{E}$ is a rank $3$ vector bundle and $\cl{L}$ is a line bundle, and viceversa. We want to find out what is the vector bundle over $\cl{H}_g$ associated to $V_{g+1}\setminus D'$.

Observe that $V_{g+1}\setminus D'$, seen as a stack in sets, has as objects the triples $(S,q,f)$ where:
\begin{itemize}
	\item $S$ is a scheme.
	\item $q$ is a global section of $\OO_{\PP^2_S}(2)$ whose zero locus $Q\subset\PP^2_S$ is smooth over $S$.
	\item $f$ is a global section of $\OO_Q(g+1)$ whose zero locus is \'{e}tale over $S$.
\end{itemize}
and $\GLt\times\Gm$ acts as described in theorem \ref{thm:presentation_Hg}. This stack is equivalent to the stack $\cl{P}$ whose objects are
$$ ((\boldsymbol{D}),\varphi,L,\sigma,\alpha)$$
where:
\begin{enumerate}
	\item  $(\boldsymbol{D})$ is a commutative diagram of the form
	$$ \xymatrix{
		C' \ar@{^{(}->}[r]^i \ar[d]^{\pi} & \PP^2_S \ar[dl] \\
		S
	}$$
	with $C'\to S$ a family of rational curves and $i$ a closed immersion. 
	\item  $\varphi:i^*\OO(1)\simeq T_{C'/S}$.
	\item $L$ is a line bundle over $C'$ of degree $-(g+1)$.
	\item $\sigma$ is a global section of $L^{-\otimes 2}$.
	\item $\alpha:\pi_*(L^{-1}\otimes T_{C'/S}^{-\otimes (g+1)/2})\simeq \OO_S$.
\end{enumerate}
The elements (1) and (2) above induce by lemma \ref{lm:triv} an isomorphism $\beta:\pi_*T_{C'/S}\simeq\OO_S^{\oplus 3}$, and vice versa. Therefore, it is easy to prove that the stack $\cl{P}$ is equivalent to the stack $\cl{P}'$ whose objects are
$$ (\pi:C'\to S,L,\sigma,\alpha,\beta)$$
where:
\begin{itemize}
	\item $\pi:C'\to S$ is a family of rational curves.
	\item $L$ is a line bundle of degree $-(g+1)$ over $C'$.
	\item $\sigma$ is a global section of $L^{-\otimes 2}$.
	\item $\alpha:\pi_*(L^{-1}\otimes T_{C'/S}^{-\otimes (g+1)/2})\simeq \OO_S$.
	\item $\beta:\pi_*T_{C'/S}\simeq\OO_S^{\oplus 3}$.
\end{itemize}
Let $\cl{H}_g^{\sim}$ be the stack whose objects are triples $(C'\to S,L,\sigma)$, where $C'\to S$ is a family of rational curves, $L$ is a line bundle on $C'$ of degree $-g-1$ and $\sigma$ is a global section of $L^{-\otimes 2}$ whose support is \'{e}tale on $S$. In \cite{ArsVis}, the authors proved that $\cl{H}_g\simeq \cl{H}_g^{\sim}$. 

There is a morphism $\cl{P}'\to\cl{H}_g^{\sim}$ defined as:
$$ (\pi:C'\to S,L,\sigma,\alpha,\beta) \longmapsto (\pi:C'\to S,L,\sigma) $$
that realizes $\cl{P}'$ as a $\GLt\times\Gm$-torsor over $\cl{H}_g^{\sim}$, because $\GLt$ acts by multiplication on $\beta$ and $\Gm$ acts by multiplication on $\alpha$.

This description of $\cl{P}'\to\cl{H}_g^{\sim}$ allows us to determine the associated rank $4$ vector bundle: it coincides with $\cl{E}\oplus\cl{L}$, where $\cl{E}$ is the rank $3$ vector bundle over $\cl{H}_g^{\sim}$ functorially defined as:
$$\cl{E}:(\pi:C'\to S,L,\sigma)\longmapsto \pi_*T_{C'/S}$$
and $\cl{L}$ is the line bundle over $\cl{H}_g^{\sim}$ functorially defined as
$$ \cl{L}:(\pi:C'\to S,L,\sigma)\longmapsto \pi_*(L^{-1}\otimes T_{C'/S}^{-\otimes (g+1)/2}) $$

We may ask for a description of the vector bundles $\cl{E}$ and $\cl{L}$ as vector bundles over $\cl{H}_g$. This can be easily deduced from the description we gave before: indeed, if $C\to S$ is a family of hyperelliptic curves of genus $g$ which is a double cover of $C'\to S$ via the morphism $\eta:C\to C'$, and if $W$ is the associated Weierstrass divisor, then:
\begin{enumerate}
	\item $\eta^*T_{C'/S}\simeq\omega_{C/S}^{-1}\otimes\OO(W)$.
	\item $\eta^*L\simeq \OO\left(-\frac{g+1}{2}W\right)$.
\end{enumerate}
From the formulas above we see that the vector bundle $\cl{E}$, seen as a vector bundle over $\cl{H}_g$, is functorially defined as
$$\cl{E}((\pi:C\to S,\iota))=\pi_*\omega_{C/S}^{-1}\left(W\right)$$ 
whereas $\cl{L}$, seen as a line bundle over $\cl{H}_g$, is functorially defined as
$$\cl{L}((\pi:C\to S,\iota))=\pi_*\omega_{C/S}^{\otimes \frac{g+1}{2}}\left(\frac{1-g}{2}W\right) $$
These considerations will be used at the end of the paper in order to provide a geometrical description of the generators of the Chow ring of $\cl{H}_g$.
\section{Intersection theory of $\PP(V_n)$}\label{sec:proj bundles}
\noindent
Let $\cl{S}$ denote the open subscheme of $\bA(2,2)$ parametrising quadratic ternary forms of rank $3$. The aim of this section is to study the vector bundles $\PP(V_n)$ on $\cl{S}$ that were introduced in the previous section (see definition \ref{def:Vn}). In  particular, in the first subsection we concentrate on the geometry of $\PP(V_n)$, and we show that over certain particular open subschemes of $\cl{S}$ the bundles $V_n$ become trivial (lemma \ref{lm:open_that_triv}). We also study some interesting morphisms between $\PP(V_n)$ for different $n$.

In the second subsection we do some computations in the $T$-equivariant Chow ring of $\PP(V_n)$, where $T\subset\GLt$ is the subgroup of diagonal matrices, focusing on the cycle classes of some specific $T$-invariant subvarieties (lemmas \ref{lm:Z_classes}, \ref{lm:Z_classes_intersection} and \ref{lm:pushforward_W_classes}). All these results will be needed for computing the Chow ring of $\cl{H}_g$, which is done in the last three sections.
\subsection{Properties of $\PP(V_n)$}\label{subsec:prop of P(Vn)}
We will use the following notational shorthand: an underlined letter $\underline{i}$ will indicate a triple $(i_0,i_1,i_2)$, and the expression $\underline{X}^{\underline{i}}$ will indicate the monomial $X_0^{i_0}X_1^{i_1}X_2^{i_2}$. A form $f$ of degree $n$ in three variables with coefficients in a ring $R$ can then be expressed as $f=\sum b_{\underline{i}}\underline{X}^{\underline{i}}$, where the $b_{\underline{i}}$ are elements of $R$ and the sum is taken over the triples $\underline{i}$ such that $|i|:=i_0+i_1+i_2=n$. The coefficients $b_{\underline{i}}$ give us coordinates in $\bA(2,n)$, thus homogeneous coordinates in $\PP(2,n)$. The symbols $a_{\underline{i}}$ will be used only for the coefficients of quadrics, or equivalently for the coordinates of $\bA(2,2)$. Finally, we say that  $\underline{i}\leq\underline{j}$ iff $i_{\alpha}\leq j_{\alpha}$ for $\alpha=0,1,2$. This is equivalent to the condition $\underline{X}^{\underline{i}}|\underline{X}^{\underline{j}}$.

Let $\cl{S}_{\underline{i}}$ be the open subscheme of $\cl{S}$ where the coordinate $a_{\underline{i}}$ is not zero, and let $Y_{\underline{i}}$ be its complement. It can be easily checked that the open subschemes $\cl{S}_{(0,2,0)}$, $\cl{S}_{(0,1,1)}$ and $\cl{S}_{(0,0,2)}$ constitute an open covering of $\cl{S}$: indeed, a point not in the union of these three open subschemes will necessarily parametrise a quadric divisible by $X_0$, thus not smooth. The open subschemes $\cl{S}_{\underline{i}}$ share another property, expressed in the following lemma:
\begin{lm}\label{lm:open_that_triv}
	The projective bundles $\PP(V_n)$ are trivial over the open sets $\cl{S}_{\underline{i}}$.
\end{lm}
The proof of the lemma above relies on a lemma of linear algebra concerning the vector spaces of forms in three variables of fixed degree.
\begin{lm}\label{lm:lin_alg_vspace_poly}
	Let $\underline{i}$ be a triple such that $|i|=2$ and let $B_n$ be the set of monomials of degree $n$ in three variables not divisible by $\underline{X}^{\underline{i}}$. Fix a quadric $q$ in three variables with non-zero coefficient $a_{\underline{i}}$. Define $B'_n$ to be the set of polynomials obtained by multiplying $q$ with a monomial of degree $n-2$ in three variables. Then the two sets $B_n$ and $B'_n$ are disjoint and $B_n \cup B'_n$ is a base for the vector space of homogeneous polynomials of degree $n$ in three variables with coefficients in a field $k$.
\end{lm}
\begin{proof}
	The fact that $B_n$ and $B'_n$ are disjoint is obvious, because every monomial in $B'_n$ is divisible by $\underline{X}^{\underline{i}}$. The monomials form a base for the vector space of homogeneous polynomials of degree $n$ in three variables. Let $M$ be the matrix representing the unique linear transformation that sends the base of monomials to the set $B_n\cup B_n'$ in the following way: monomials not divisible by $\underline{X}^{\underline{i}}$ are sent to themselves, and monomials of the form $\underline{X}^{\underline{i}}f$ are sent to $qf$. Observe that, after possibly reordering the monomials of the form $\underline{X}^{\underline{i}}f$, the matrix $M$ will be upper triangular, with either $1$ or $a_{\underline{i}}$ on the diagonal: this shows that the determinant of $M$ is invertible, thus $B'_n\cup B_n$ is a base. 
\end{proof}
%\begin{proof}[Proof of lemma \ref{lm:open_that_triv}]
%We want to produce an isomorphism $V_n|_{\cl{S}_{\underline{i}}}\simeq \OO_{\cl{S}_{\underline{i}}}^{\oplus 2n+1}$. Consider the morphism
%$$ \OO_{\cl{S}_{\underline{i}}}^{\oplus (n+2)(n+1)/2}\simeq \bA_{\cl{S}_{\underline{i}}}(2,n)\longrightarrow \OO_{\cl{S}_{\underline{i}}}^{\oplus 2n+1},\quad f\longmapsto (b_{\underline{k}}(f))$$
%where the triples $\underline{k}$ satisfy $|\underline{k}|=n$ and $\underline{i}\nleq\underline{k}$. This morphism is surjective and lemma \ref{lm:lin_alg_vspace_poly} implies that its kernel is exactly the image of the morphism $$\bA_{\cl{S}_{\underline{i}}}(2,n-2)\longrightarrow\bA_{\cl{S}_{\underline{i}}}(2,n)$$
%that we introduced in the first section, when we defined the vector bundles $V_n$. From this we deduce the isomorphism $V_n|_{\cl{S}_{\underline{i}}}\simeq \OO_{\cl{S}_{\underline{i}}}^{\oplus 2n+1}$.
%\end{proof}
From lemma \ref{lm:lin_alg_vspace_poly} we see that over $S_{\underline{i}}$ the coordinates $b_{\underline{k}}$, for $|k|=n$ and $\underline{i}\nleq\underline{k}$, trivialize the vector bundle $V_n|_{\cl{S}_{\underline{i}}}$, thus proving lemma \ref{lm:open_that_triv}.
We now define some morphisms that will play an important role in the remainder of the paper: these morphisms are
$$\pi_{n,m}:\PP(1,2n)\times\PP(1,2m)\longrightarrow \PP(1,4n+2m),\quad (f,g)\longmapsto f^2g $$
whose $\GLt$-counterparts are
$$ \pi'_{n,m}:\PP(V_n)\times_{\cl{S}}\PP(V_m)\longrightarrow\PP(V_{2n+m}),\quad (q,f,g)\longmapsto (q,f^2g) $$
Applying theorem \ref{cor:chow diagrams} we deduce the following commutative diagrams of equivariant Chow rings:
$$\xymatrix{
	A^*_{\GLt}(\PP(V_n))\otimes_{A^*_{\GLt}(\cl{S})}A^*_{\GLt}(\PP(V_m)) \ar[r]^{\hspace{1cm}\pi'_{n,m*}} \ar[d]^*[@]{\cong} & A^*_{\GLt}(\PP(V_{2n+m})) \ar[d]^*[@]{\cong} \\
	A^*_{\PGLt}(\PP(1,2n))\otimes_{A^*_{\PGLt}} A^*_{\PGLt}(\PP(1,2m)) \ar[r]^{\hspace{1.2cm}\pi_{n,m*}} & A^*_{\PGLt}(\PP(1,4n+2m)) }$$

%\begin{rmk}
%	The morphism $\PP_{\cl{S}}(2,1)\times_{\cl{S}}\cdots\times_{\cl{S}}\PP_{\cl{S}}(2,1)\to \PP_{\cl{S}}(2,n)$ is not surjective, because its image is the locus of polynomials that can be expressed as a product of linear factors, which is obviously not always the case if we deal with homogeneous polynomials in more than two variables. This shows that it is really necessary to work with the non-trivial projective bundles of the form $\PP(V_n)$, in order to have surjectivity.
%\end{rmk}
Observe that we also have the following class of closed linear immersions of projective bundles:
$$j_{n,r,l}:\PP(V_{n-r-l})\hookrightarrow \PP(V_n),\quad(q,f)\longmapsto (q,X_0^rX_1^lf) $$
These embeddings are not $\GLt$-equivariant but only $T$-equivariant, where $T\subset \GLt$ is the maximal subtorus of diagonal matrices.
\begin{df}\label{def:Wnrl}
	We define the $T$-invariant, closed subscheme $W_{n,r,l}\subset\PP(V_n)$ as the image of $j_{n,r,l}$.
\end{df}
\subsection{Computations in the $T$-equivariant Chow ring of $\PP(V_n)$}
Let us introduce another little piece of notation: with $\underline{\lambda}$ we mean the triple $(\lambda_1,\lambda_2,\lambda_3)$. If $\underline{i}$ is another triple (most of the times, we will have $\underline{i}=(i_0,i_1,i_2)$), we indicate with $\underline{i}\cdot\underline{\lambda}$ their scalar product.

As already observed, on $\PP(V_n)$ there is a well defined action of $\GLt$, which induces an action of its maximal subtorus $T$ of diagonal matrices. Thus we can consider the $T$-equivariant Chow ring $A^*_T(\PP(V_n))$. If $Z\subset \PP(V_n)$ is a $T$-invariant subvariety, its $T$-equivariant cycle class will be denoted $[Z]$.

Recall that, as stated in the introduction, when dealing with $T$-equivariant Chow groups we denote $\lambda_i$, for $i=1,2,3$, the generators of $A^T({\rm Spec}(k))$ and we denote $c_i$ the elementary symmetric polynomials in $\lambda_1$, $\lambda_2$, $\lambda_3$.
\begin{lm}\label{lm:T-Chow of S}
	We have $A^*_T(\cl{S})=\ZZ[\lambda_1,\lambda_2,\lambda_3]/(c_1,2c_3)$.
\end{lm}
\begin{proof}
	Observe that $\cl{S}$ is an open subscheme of the representation $\bA(2,2)$ of $T$. This plus the localization exact sequence implies that $A_T^*(\cl{S})$ is generated by $\lambda_1$, $\lambda_2$ and $\lambda_3$. We need to find the relations among the generators.
	Let $\cl{W}$ be the closed subscheme inside $\PP(2,2)\times\PP^2$ whose points correspond to pairs $(q,p)$ with $p$ a singular point of $\{q=0\}$. We have
	$ [\cl{W}]={\rm pr}_1^*\xi_3+{\rm pr}_1^*\xi_2t+{\rm pr}_1^*\xi_1t^2 $, 
	where $t$ denotes the pullback to $\PP(2,2)\times\PP^2$ of the hyperplane section of $\PP^2$. Using the same arguments of \cite{FulVis}*{theorem $5.5$} we deduce that the image of
	$$ i_*:A_T^*(\PP(2,2)_{\rm sing})\longrightarrow A_T^*(\PP(2,2)) $$
	is exactly $(\xi_1,\xi_2,\xi_3)$. Arguing as in \cite{Vis}*{pg.638}, we have that the pullback map
	$$ A^*_T(\PP(\cl{S}))\twoheadrightarrow A^*_T(\cl{S}) $$
	is surjective with kernel $(c_1-h)$, where $h$ denotes the hyperplane section of $\PP(\cl{S})$. This implies that the ideal of relations of $A^*_T(\cl{S})$ is generated by the elements $\xi_1,\xi_2,\xi_3,c_1-h$ and $f(h)$, where $f$ denotes the monic polynomial of degree $6$ satisfied by $h$ in $A^*_T(\PP(2,2))$. 
	
	The subscheme $\cl{W}$ is a complete intersection of three $T$-invariant hyperplanes, namely those defined as the vanishing locus of $q_{X_i}(p)$ for $i=0,1,2$.
	We can use the formula given by \cite{EdiFul}[lemma 2.4] to compute the cycle class of each hypersurface, whose product will be equal to $[\cl{W}]$.
	From the result of the explicit computation we see that the actual generators of the ideal of relations are $c_1$ and $2c_3$.
\end{proof}
We know that $A^*_T(\PP(V_n))$ is generated as $A^*_T(\cl{S})$-algebra by the hyperplane section $h_n$, so that we have
$$ A^*_T(\PP(V_n))\simeq\mathbb{Z}[\lambda_1,\lambda_2,\lambda_3,h_n]/(c_1,2c_3,p_n(h_n)) $$
where $p_n(h_n)$ is a monic polynomial of degree $2n+1$. Recall now that in the previous subsection we defined the open subscheme $\cl{S}_{\underline{i}}$ of $\cl{S}$ as the subscheme whose points are smooth quadratic forms with coefficient $a_{\underline{i}}$ not zero. The complement of $\cl{S}_{\underline{i}}$ was denoted $Y_{\underline{i}}$. The $T$-equivariant Chow ring of $\PP(V_n)|_{\cl{S}_{\underline{i}}}$ may be easily computed.
\begin{lm}\label{lm:Chow_restricted_projBundle}
	We have $A^*_T(\PP(V_n)|_{\cl{S}_{\underline{i}}})\simeq A^*_T(\PP(V_n))/(\underline{i}\cdot\underline{\lambda})$.
\end{lm}
\begin{proof}
	From the localization exact sequence
	$$A^*_T(\PP(V_n)|_{Y_{\underline{i}}})\xrightarrow{i_*} A^*_T(\PP(V_n)) \xrightarrow{j^*} A^*_T(\PP(V_n)|_{\cl{S}_{\underline{i}}}) \rightarrow 0$$
	we see that what we need to prove is that ${\rm im}(i_*)=(\underline{i}\cdot\underline{\lambda})$. Let $t$ be the hyperplane section of $\PP(V_n)|_{Y_{\underline{i}}}$, so that the $T$-equivariant Chow ring of $\PP(V_n)|_{Y_{\underline{i}}}$ is generated, as abelian group, by elements of the form $p^{'*}\alpha\cdot t^d$ for $d\leq 2n$, where $p'$ is the projection map to $Y_{\underline{i}}$. This implies that ${\rm im}(i_*)$ is generated by the elements $i_*(p^{'*}\alpha\cdot t^d)$. Observe that $i^*h_n=t$. From the cartesian square
	$$ \xymatrix{ 
		\PP(V_n)|_{Y_{\underline{i}}} \ar[r]^{i} \ar[d]^{p'} & \PP(V_n) \ar[d]^{p} \\
		Y_{\underline{i}} \ar[r]^{i'} & \cl{S} } $$
	we deduce the following chain of equivalences:
	\begin{align*}
	i_*(p^{'*}\alpha\cdot t^d)&=i_*(p^{'*}\alpha\cdot i^*h_n^d)\\
	&=i_*p^{'*}\alpha\cdot h_n^d=p^*i'_*\alpha\cdot h^d_n
	\end{align*}
	This means that ${\rm im}(i_*)=p^*({\rm im}(i'_*))$. Now observe that $Y_{\underline{i}}$ is an open subscheme of a representation of $T$, namely the vector subspace of forms of degree $2$ with the coefficient $a_{\underline{i}}$ equal to zero. This implies that $A^*_T(Y_{\underline{i}})$ is a quotient of $A^*_T$, from which we deduce that ${\rm im}(i'_*)$ is generated, as an ideal, by $[Y_{\underline{i}}]$, so that we only have to compute this class in $A^*_T(\cl{S})$. Because of the fact that $Y_{\underline{i}}$ is defined as the zero locus of the coordinate $a_{\underline{i}}$, the class of $Y_{\underline{i}}$ corresponds to $-c_1(\chi)$: here $\chi$ is the character such that if $\tau$ is an element of $T$ then $a_{\underline{i}}(\tau^{-1}\cdot x)=\chi(\tau)a_{\underline{i}}$. In this case, with the action that we defined before, we obtain $c_1(\chi)=c_1-\underline{i}\cdot\underline{\lambda}$. Using the relations that we already had in $A^*_T(\cl{S})$, we deduce that $[Y_{\underline{i}}]=-\underline{i}\cdot\underline{\lambda}$.
\end{proof}
Recall that we previously defined $T$-invariant closed subschemes $W_{n,r,l}\subset\PP(V_n)$ (see definition \ref{def:Wnrl}).
The cycle classes $[W_{n,r,l}]$ have degree $2r+2l$. We can pull back them via the open immersion 
$$j:\PP(V_n)|_{\cl{S}_{(0,0,2)}}\simeq \PP^{2n}\times \cl{S}_{(0,0,2)}\hookrightarrow \PP(V_n)$$
so that they can be actually computed. Indeed here we have homogeneous coordinates given by the coefficients $b_{\underline{k}}$, for $|k|=n$ and $(0,0,2)\nleq \underline{k}$ and we see that
$$ j^{-1}W_{n,r,l}=\{b_{\underline{k}}=0\text{ for }k_0<r\text{ or }k_1<l \} $$
from which we deduce that they are complete intersection. Using the formula of \cite{EdiFul}[lemma 2.4], we obtain
$$j^*[W_{n,r,l}]=\prod (h_n-\underline{k}\cdot\underline{\lambda})\text{ for }\underline{k}\text{ s.t. } |\underline{k}|=n\text{, }k_2<2\text{, }k_0<r\text{ or }k_1<l$$
Applying lemma \ref{lm:Chow_restricted_projBundle} we deduce:
\begin{lm}\label{lm:Z_classes}
	We have  $$[W_{n,r,l}]=\prod (h_n-\underline{k}\cdot\underline{\lambda})+2\lambda_3\xi \text{ for }\underline{k}\text{ s.t. } |\underline{k}|=n\text{, }k_2<2\text{, }k_0<r\text{ or }k_1<l$$
	where $\xi$ is an element of $A^*_T(\PP(V_n))$.
\end{lm}
In particular, all these classes are monic in $h_n$. Another useful property is the following: the set-theoretic intersection of $W_{n,r,0}$ and $W_{n,0,l}$ is exactly $W_{n,r,l}$. This is also true at the level of Chow rings: indeed $W_{n,r,l}$ is the only component of $W_{n,r,0}\cap W_{n,0,l}$, all the varieties involved are smooth and it is easy to check that the intersection is transversal, so that:
\begin{lm}\label{lm:Z_classes_intersection}
	We have $[W_{n,r,0}]\cdot [W_{n,0,l}]=[W_{n,r,l}]$.
\end{lm}
%Recall the following commutative square of proper morphisms
Recall that we have defined the morphism
%$$ \xymatrix{
%	\PP(V_1)^{\times_{\cl{S}} n}\times_{\cl{S}} \PP(V_1)^{\times_{\cl{S}} m} \ar[r] \ar[d]^{\rho_n \times \rho_m} & \PP(V_1)^{\times_{\cl{S}} 2n+m} \ar[d]^{\rho_{2n+m}} \\
%	\PP(V_n)\times_{\cl{S}}\PP(V_m) \ar[r]^{\psi'_{n,m}} & \PP(V_{2n+m}) }$$
$$ \pi'_{n,m}:\PP(V_n)\times_{\cl{S}}\PP(V_m)\longrightarrow\PP(V_{2n+m})$$
as $\pi'_{n,m}(q,f,g)=(q,f^2g)$.
%Let us call $P_\alpha$ the image of the section $$\cl{S}\to\PP(V_1),\quad q\mapsto (q,X_\alpha)$$ for $\alpha=0,1$. Then from the square above we obtain another commutative square, which is
If we restrict this morphism to $W_{n,1,n-1}\times_{\cl{S}}\PP(V_{m})$ we obtain a birational, proper and surjective morphism
%$$ \xymatrix{
%	P_0\times_{\cl{S}} P_1\times_{\cl{S}}\cdots\times_{\cl{S}} P_1 \times_{\cl{S}} \PP(V_1)^{\times_{\cl{S}} m} \ar[r] \ar[d]^{\rho_n \times \rho_m} & P_0\times_{\cl{S}} P_0 \times_{\cl{S}} P_1\times_{\cl{S}}\cdots\times_{\cl{S}} P_1 \times_{\cl{S}}\PP(V_1)^{\times_{\cl{S}} m} \ar[d]^{\rho_{2n+m}} \\
%	W_{n,1,n-1}\times_{\cl{S}}\PP(V_m) \ar[r]^{\psi'_{n,m}}& W_{2n+m,2,n-2} }$$
$$\pi_{n,m}':W_{n,1,n-1}\times_{\cl{S}}\PP(V_{m})\longrightarrow  W_{2n+m,2,2n-2}$$
and from this we deduce:
%Observe that both the left and the right vertical morphisms are surjective and finite of degree $(2m-1)!!$. This, plus the commutativity of the square, implies that $(2m-1)!!\psi'_{n,m*}([W_{n,1,n-1}]\times 1)=(2m-1)!! [W_{2n+m,2,n-2}]$, where $[W_{n,1,n-1}]\times 1$ denote the cycle class $[W_{n,1,n-1}\times_{\cl{S}}\PP(V_m)]$. We can clear out $(2m-1)!!$ from both sides of the equality, because $(2m-1)!!$ is odd and the integers of torsion in $A^*_T(\PP(V_n))$ are even. We have proved the
\begin{lm}\label{lm:pushforward_W_classes}
	We have $\pi'_{n,m*}([W_{n,1,n-1}\times_{\cl{S}}\PP(V_{m})])=[W_{2n+m,2,n-2}]$.
\end{lm}
\section{The Chow ring of $\cl{H}_g$: generators and first relations}\label{sec:chow 1}
\noindent
The goal of this section is to do the first steps in the computation of the Chow ring of $\cl{H}_g$, finding the generators and some relations. The intermediate result that we find is the content of corollary \ref{cor:chow_ring_Hg_intermediate}.

Let $V_{g+1}$ be the vector bundle over $\cl{S}$ introduced in definition \ref{def:Vn}, where $\cl{S}$ is the scheme of quadratic ternary forms of rank $3$. Recall (see remark \ref{rmk:point of Vn}) that for any field $K$ the points of $V_{g+1}$ can be thought as pairs $(q,f)$ where $q$ is a quadratic ternary form of rank $3$ over $K$ and $f$ is an equivalence class of ternary forms of degree $g+1$ over $K$, where two forms are equivalent if their difference is divisible by $q$.

Theorem \ref{thm:presentation_Hg} tells us that the Chow ring $A^*(\cl{H}_g)$ is isomorphic to the equivariant Chow ring $A^*_{\GLt\times\Gm}(U')$, where $U'$ is the open subscheme of $V_{g+1}$ consisting of all the pairs $(q,f)$ such that the intersection $Q\cap F$ of the associated plane curves is smooth, i.e. the intersection consists of $2g+2$ distinct points.

Let $D'$ be the complement of $U'$ in $V_{g+1}$, which is a closed subscheme of codimension $1$. In this section, we will always assume $n=g+1$. The localization exact sequence in this case is
$$ A^*_{\GLt\times\Gm}(D')\xrightarrow{i_*} A^*_{\GLt\times\Gm}(V_n) \xrightarrow{j^*} A^*_{\GLt\times\Gm}(U') \rightarrow 0 $$
Observe that $V_n$ is a $\GLt\times\Gm$-equivariant vector bundle over $\cl{S}$, from which we deduce
\begin{equation}\label{eq:chow S}
A^*_{\GLt\times\Gm}(V_n) \simeq A^*_{\GLt\times\Gm}(\cl{S}) \simeq \mathbb{Z}[\tau,c_2,c_3]/(2c_3) 
\end{equation}
where $\tau$ is the first Chern class of the standard $\Gm$-representation and $c_2$, $c_3$ are respectively the second and the third Chern class of the standard $\GLt$-representation. The last isomorphism can be deduced from lemma \ref{lm:T-Chow of S} using the fact that the $\GLt\times\Gm$-equivariant Chow ring of $\cl{S}$ is the $S_3$-invariant subring of $A^*_{T\times\Gm}(\cl{S})$, where $S_3$ acts by permuting $\lambda_1$, $\lambda_2$ and $\lambda_3$, the cycles pulled back from $A^*_T({\rm Spec}(k))$ (see \cite[lemma $2.1$]{FulVis}).

We have found a set of generators for $A^*(\cl{H}_g)$. Now we have to find the relations among the generators, which is the same as computing the generators of the ideal ${\rm im}(i_*)$ inside the equivariant Chow ring of $V_n$. 

Consider the projective bundle $\PP(V_n)$ and its open subscheme $U$, whose preimage in $V_n\setminus\sigma_0$ is exactly $U'$ (recall that $\sigma_0:\cl{S}\to V_n$ is the zero section). Observe that $V_n\setminus\sigma_0$ is equivariantly isomorphic to the $\Gm$-torsor over $\PP(V_n)$ associated to the line bundle $\OO(-1)\otimes \cl{V}^{-\otimes 2}$, where $\cl{V}$ is the standard representation of $\Gm$ pulled back to $\PP(V_n)$. Clearly, we can say the same thing for $U'$ over $U$. This implies, arguing as in \cite{Vis}*{pg.638}, that we have a surjective morphism
$$A_{\GLt\times\Gm}^*(U)\longrightarrow A_{\GLt\times\Gm}^*(U') $$
whose kernel is given by $c_1(\OO_U(-1)\otimes(\cl{V})^{-\otimes 2})=-h_n-2\tau$. 

This means that, if $p(h_n)$ is a relation in $A^*_{\GLt\times\Gm}(U)$, then $p(-2\tau)$ is a relation in $A^*_{\GLt\times\Gm}(U')$, and all the relations in this last ring are obtained from relations in the Chow ring of $U$ in this way.

After passing to $\PP(V_n)$ the action of $\Gm$ becomes trivial, so that we can restrict ourselves to consider only the $\GLt$-action. Is then enough, in order to determine the equivariant Chow ring of $U'$, to compute the $\GLt$-equivariant Chow ring of $U$. 

Again, we have the localization exact sequence
$$ A^*_{\GLt}(D)\xrightarrow{i_*}A^*_{\GLt}(\PP(V_n))\xrightarrow{j^*}A^*_{\GLt}(U)\rightarrow 0 $$
The ring in the middle is isomorphic to 
\begin{equation}\label{eq:chow P(Vn)}
A^*_{\GLt}(\cl{S})[h_n]/(p_n(h_n))\simeq\mathbb{Z}[c_2,c_3,h_n]/(2c_3,p_n(h_n))
\end{equation}
where $p_n(h_n)$ is monic of degree $2n+1$. Thus to find the relations we have to compute the generators of the ideal ${\rm im}(i_*)$. 

Observe that the closed subscheme $D$ admits a stratification $$ D_n \subset D_{n-1} \subset ... \subset D_1=D $$ where
$D_m$ is the locus of pairs $(q,f)$ such that $Q\cap F=2E+E'$, with $\deg (E)=m$. All these sets are clearly $\GLt$-invariant. Observe moreover that $D_{2s}$ coincides with the image of the equivariant, proper morphism
$$\pi'_{2s}:\PP(V_s)\times_{\cl{S}}\PP(V_{n-2s})\longrightarrow \PP(V_n),\quad (q,f,g)\longmapsto (q,f^2g)$$
where $\pi'_{2s}:=\pi'_{s,n-2s}$, the morphism defined in subsection \ref{subsec:prop of P(Vn)}. This induces a scheme structure on $D_{2s}$. 

Consider also the $\GLt$-invariant closed subscheme
$$ \cl{Y}_1=\{(q,l)\text{ such that }L\text{ is tangent to }Q\} \subset \PP(V_1)$$
and let us define the closed subschemes $\cl{Y}_{2s+1}$ as the image of the morphisms
$$\phi:\cl{Y}_1\times_{\cl{S}}\PP(V_s)\longrightarrow \PP(V_{2s+1}),\quad (q,l,f)\longmapsto (q,lf^2)$$
We can think of $\cl{Y}_{2s+1}$ as the locus of quadrics plus a divisor of the form $2E$, with $\deg(E)=2s+1$. Restricting the morphism $\psi'_{2s+1,n-2s-1}$ defined in subsection \ref{subsec:comp res} to this closed subscheme we obtain a proper morphism 
$$\pi'_{2s+1}:\cl{Y}_{2s+1}\times_{\cl{S}}\PP(V_{n-2s-1})\longrightarrow\PP(V_n)$$
whose image is $D_{2s+1}$. This induces the scheme structure on $D_{2s+1}$. 

The stratification defined above resembles the stratification 
$$\Delta_n\subset...\subset\Delta_1=\Delta\subset\PP(1,2n)$$
that has been introduced in \cite{FulViv}. Indeed, we have that $D_s$ is the $\GLt$-counterpart of $\Delta_s$. Furthermore, the $\GLt$-counterparts of the morphisms
$$\pi_{2s}:\PP(1,2s)\times\PP(1,2n-4s)\longrightarrow\PP(1,2n),\quad (f,g)\longmapsto f^2g$$
are exactly the morphisms $\pi'_{2s}:\PP(V_{s})\times_{\cl{S}}\PP(V_{n-2s})\to\PP(V_n)$. Applying again theorem \ref{cor:chow diagrams} we obtain the commutative diagram
$$ \xymatrix{ A^*_{\PGLt}(\PP(1,2s)\times\PP(1,2n-4s)) \ar[d] \ar[r]^{\simeq} & A^*_{\GLt}(\PP(V_{s})\times_{\cl{S}}\PP(V_{n-2s})) \ar[d] \\
	A^*_{\PGLt}(\PP(1,2n))  \ar[r]^{\simeq} & A^*_{\GLt}(\PP(V_n)) }$$
We also have that $\cl{Y}_1$ is the $\GLt$-counterpart of $\PP^1$, as we can think of $\cl{Y}_1$ as the tautological conic over $\cl{S}$, and in general the closed subschemes $\cl{Y}_{2s+1}\subset \PP(V_{2s+1})$ are the $\GLt$-counterpart of $\PP(1,2s+1)$ sitting inside $\PP(1,4s+2)$ via the square map, so that we still have
$$ \xymatrix{
	A^*_{\PGLt}(\PP(1,2s+1)\times\PP(1,2n-4s-2)) \ar[r]^{\simeq} \ar[d] & A^*_{\GLt}((\cl{Y}_{2s+1})\times_{\cl{S}}\PP(V_{n-2s-1})) \ar[d] \\
	A^*_{\PGLt}(\PP(1,4s+2)\times\PP(1,2n-4s-2))  \ar[r]^{\simeq} \ar[d]&   A^*_{\GLt}(\PP(V_{2s+1})\times_{\cl{S}}\PP(V_{n-2s-1})) \ar[d] \\
	A^*_{\PGLt}(\PP(1,2n)) \ar[r]^{\simeq} & A^*_{\GLt}(\PP(V_n)) } $$
Combining \cite{FulViv}*{lemma 3.1} with the diagrams above we obtain:
\begin{lm}\label{lm:chow_envelope}
	The ideal ${\rm im}(i_*)$ is the sum of the ideals ${\rm im}(\pi'_{s*})$, and these ideals are equal to ${\rm im}(\pi_{s*})$ via the isomorphisms of the diagrams above.
\end{lm}
In particular, lemma \ref{lm:chow_envelope} and \cite{FulViv}*{prop. 5.2} imply:
\begin{prop}\label{pr:im_pi_1}
	We have ${\rm im}(\pi'_{1*})=(2h_n^2-2n(n-1)c_2,(4n-2)h_n)$.
\end{prop}
\begin{rmk}
	The content of proposition \ref{pr:im_pi_1} can be proved without relying on \cite{FulViv}*{prop. 5.2}. In order to make the present work shorter, we decided not to write here the alternative proof, which can be found in the author's Ph.D. thesis \cite{Dil}.
\end{rmk}
\begin{cor}\label{cor:chow_ring_Hg_intermediate}
	The Chow ring of $\cl{H}_g$ is a quotient of the ring $$\mathbb{Z}[\tau,c_2,c_3]/(4(2g+1)\tau,8\tau^2-2g(g+1)c_2,2c_3)$$
\end{cor}
\begin{proof}
	The only thing we need to prove is that $p_n(-2\tau)$ is contained in the ideal above. This works exactly as in \cite{FulViv}*{proposition 6.4}.
\end{proof}
The next section will be devoted to check if there are other relations in the Chow ring of $\cl{H}_g$, or if they all come from the pullback to $A^*_{\GLt\times\Gm}(V_n)$ of ${\rm im}(\pi'_{1*})$.
\section{Other generators of ${\rm im}(i_*)$}\label{sec:chow 2}
\noindent
In this section, we first complete the computation of ${\rm im}(i_*)$ with $\ZZ[\frac{1}{2}]$-coefficients, which means that we do the computations in the equivariant Chow ring tensored over $\ZZ$ with $\ZZ[\frac{1}{2}]$. The main result is the following proposition:
%\begin{prop}\label{pr:Chow ring H_g with Z-half coefficients}
%	We have ${\rm im}(i_*)\otimes_{\ZZ}\ZZ[\frac{1}{2}]=(2h_n^2-2n(n-1)c_2,4(n-2)h_n)$.
%\end{prop}
\begin{prop}\label{pr:generators_im i_Zhalf coefficients}
	We have ${\rm im}(\pi'_{r*})\otimes_{\mathbb{Z}}\mathbb{Z}[\frac{1}{2}]\subset {\rm im}(\pi'_{1*})	\otimes_{\mathbb{Z}}\mathbb{Z}[\frac{1}{2}]$.
\end{prop}
In other terms, using $\ZZ[\frac{1}{2}]$-coefficients, the ideal ${\rm im}(i_*)$ coincides with the ideal ${\rm im}(\pi'_{1*})$, whose generators we computed in the previous section (proposition \ref{pr:im_pi_1}).

Next, we pass to $\ZZ_{(2)}$-coefficients. What we deduce at the end is the following result:
\begin{prop}\label{pr:generators im i with Z-two coefficients}
	We have ${\rm im}(i_*)\otimes_{\ZZ}\ZZ_{(2)}=(2h_n^2-2n(n-1)c_2,4(n-2)h_n,\pi'_{2*}(h_1^2\times [\PP(V_{n-2})]))$ and the inclusion ${\rm im}(\pi'_{1*})\otimes_{\ZZ}\ZZ_{(2)}\subset {\rm im}(i_*)\otimes_{\ZZ}\ZZ_{(2)}$ is strict.
\end{prop}
The last two propositions together imply:
\begin{cor}\label{cor:generators im i}
	We have ${\rm im}(i_*)=(2h_n^2-2n(n-1)c_2,4(n-2)h_n,\pi'_{2*}(h_1^2\times [\PP(V_{n-2})]))$.
\end{cor}
Using theorem \ref{cor:chow diagrams} we see that the corollary above, interpreted in the $\PGLt$-equivariant setting, says that the image of
$$ i_*:A^*_{\PGLt}(\Delta)\longrightarrow A^*_{\PGLt}(\PP(1,2n))$$
is equal to the image of
$$ \pi_{1*}:A^*_{\PGLt}(\PP(1,1)\times\PP(1,2n-2))\longrightarrow A^*_{\PGLt}(\PP(1,2n))$$
plus the cycle $\pi_{2*}(H^2\times 1)$, where $H$ is the hyperplane section of $\PP(1,2)$ and the morphism $\pi_2$ is
$$ \pi_2:\PP(1,2)\times\PP(1,2n-4)\longrightarrow \PP(1,2n),\quad (f,g)\longmapsto f^2g$$
In particular, the inclusion ${\rm im}(\pi_{1*})\subset {\rm im}(i_*)$ is strict. Instead, in \cite{FulViv}*{proposition 5.3} was stated that ${\rm im}(\pi_{r*})\subset {\rm im}(\pi_{1*})$.

To prove proposition \ref{pr:generators im i with Z-two coefficients}, we initially work with $\GLt$-equivariant Chow rings, but then we start using the $T$-equivariant ones, and we complete the computation of ${\rm im}(i_*^T)\otimes_{\ZZ}\ZZ_{(2)}$ in this different setting. Then, using what we have found exploiting the $T$-equivariant Chow rings, we go back to the $\GLt$-equivariant setting and we finish the computation of the generators of ${\rm im}(i_*)\otimes_{\ZZ}\ZZ_{(2)}$ .

We initially follow the path of \cite{FulViv} but at a certain point we diverge. Indeed, as said before, the computation is completed in the $T$-equivariant setting, mainly because we start working with cycle classes of subvarieties that are only $T$-invariant and not $\GLt$-invariant. In particular, these classes do not have an analogue in the $\PGLt$-equivariant setting that is adopted in \cite{FulViv}. This is where we really need the new presentation given by theorem \ref{thm:presentation_Hg}.
\subsection{Computations with $\mathbb{Z}[\frac{1}{2}]$-coefficients}
Proposition \ref{pr:generators_im i_Zhalf coefficients} is a direct consequence of lemma \ref{lm:chow_envelope} and \cite{FulViv}*{subsec. 5.7}. In the author's Ph.D. thesis \cite{Dil} it can be found a direct proof of proposition  \ref{pr:generators_im i_Zhalf coefficients} which does not rely on \cite{FulViv}.
%In order to obtain the relations for the Chow ring of $\cl{H}_g$, we use the morphism
%$$ A^*_{\GLt}(U)\twoheadrightarrow A^*_{\GLt}(U') $$
%introduced in the third section.  Tensoring with $\ZZ[\frac{1}{2}]$ does not affect in any way the surjectivity nor the kernel, so that in order to compute $A^*(\cl{H}_g)\otimes_{\ZZ}\ZZ[\frac{1}{2}]$ we need simply to add the relation $-2\tau-h_n=0$. Putting together proposition \ref{pr:generators_im i_Zhalf coefficients} and corollary \ref{cor:chow_ring_Hg_intermediate}, we readily deduce proposition \ref{pr:Chow ring H_g with Z-half coefficients}.
\subsection{Computations with $\ZZ_{(2)}$-coefficients, first part}
Throughout this subsection, we will assume that every Chow ring and every ideal appearing is tensored with $\ZZ_{(2)}$. %Moreover, we will work with both $T$-equivariant and $\GLt$-equivariant Chow rings. 
We start by recalling the content of \cite{FulViv}*{lemma $5.4$}:
\begin{lm}\label{lm:ker_is_torsion}
	Let $X$ be a smooth scheme on which ${\rm PGL}_n$ acts, and consider the induced action of ${\rm SL}_n$ via the quotient map ${\rm SL}_n\to{\rm PGL}_n$. Then the kernel of the pullback map $A^*_{{\rm PGL}_n}(X)\to A^*_{{\rm SL}_n}(X)$ is of $n$-torsion.
\end{lm}
The main result is the next proposition.
\begin{prop} \label{pr:generators first reduction}
	The ideal ${\rm im}(i_*)\otimes_{\mathbb{Z}}\ZZ_{(2)}$ is equal to the sum of the ideal ${\rm im}(\pi'_{1*})\otimes_{\mathbb{Z}}\ZZ_{(2)}$ and the ideal generated by the elements $\pi'_{2s*}(h_s^{2s}\cdot 1)$ for $s=1,...,n/2$
\end{prop}
%\begin{prop}\label{pr:generators_im i_Z2 coefficients}
%We have the following equivalence of ideals inside $A^*_T(\PP(V_n))$:$${\rm im}(i_*)\otimes_{\ZZ}\ZZ_{(2)}=(2h_n^2-2n(n-1)c_2,4(n-2)h_n,[W_{n,2,0}])\otimes_{\ZZ}\ZZ_{(2)}$$ Moreover, the ideal ${\rm im}(\pi'_{1*})\otimes_{\ZZ}\ZZ_{(2)}$ is strictly contained inside ${\rm im}(i_*)\otimes_{\ZZ}\ZZ_{(2)}$.
%\end{prop}
Recall from lemma \ref{lm:chow_envelope} that the ideal ${\rm im}(i_*)$ is the sum of the ideals ${\rm im}(\pi'_{r*})$. The following lemma is an immediate consequence of lemma \ref{lm:chow_envelope} and \cite{FulViv}*{lemma $5.5$}. An alternative proof in the $\GLt$-equivariant setting can be found in \cite{Dil}.
\begin{lm}\label{lm:odd case Z2 coefficients}
	We have ${\rm im}(\pi'_{2s+1*})\otimes_{\mathbb{Z}}\ZZ_{(2)}\subset{\rm im}(\pi'_{1*})\otimes_{\ZZ} \ZZ_{(2)}$.
\end{lm}
Now we want to study the image of $\pi'_{2s*}$. Observe that this ideal is generated by the pushforward of the classes $h_{2s}^i\cdot h_{n-2s}^j$, for $i=0,...,2s$ and $j=0,...,2n-4s$, where we use the notational shorthand $h_{2s}^i\cdot h_{n-2s}^j$ to indicate what should be more correctly denoted as ${\rm pr}_1^*h_{2s}^i\cdot{\rm pr}_2^* h_{n-2s}^j$. An intermediate result is the following lemma:
\begin{lm}\label{lm:not highest degree classes are in I}
	We have that $\pi'_{2s*}(h_s^i\cdot h_{n-2s}^j)$ is in ${\rm im}(\pi'_{1*})$ for $i=0,...,2s-1$ and $j=0,...,4n-2s$.
\end{lm}
In order to prove the lemma above we need a technical result, which can be found also in \cite{FulViv}, with the exception that there the authors claim the result also for $i=2s$. The proof works exactly in the same way.
\begin{lm}\label{lm:2-divisibility}
	We have that $\pi'_{2s*}(h_s^i\cdot h_{n-2s}^j)$ is $2$-divisible for $i=0,...,2s-1$ and $j=0,...,4n-2s$.
\end{lm}
\begin{proof}
	We start with the case $s=n/2$. Observe that $\pi'_{n*}(h_{n/2}^i)$, for $i=0,...,n-1$, is $2$-divisible if and only if $\pi'_{n*}(h_{n/2}^i)\cdot h_n$ is $2$-divisible. This follows from the uniqueness of the representation of cycles in $A^*_T(\PP(V_n))$ as polynomials in $h_n$ of degree less or equal to $2n$. We also have that $\pi^{'*}_n h_n=2h_{n/2}$, and from this we deduce that
	$$\pi'_{n*}(h_{n/2}^i)\cdot h_n=\pi'_{n*}((h_{n/2}^i)\cdot \pi^{'*}_n h_n)=2\pi'_{n*}h_{n/2}^{i+1}$$
	Now consider the general case, and observe that we have a factorization of $\pi'_{2s}$ as follows:
	$$ \PP(V_s)\times_{\cl{S}}\PP(V_{n-2s})\xrightarrow{\pi''_{2s}\times{\rm id}} \PP(V_{2s})\times_{\cl{S}}\PP(V_{n-2s})\xrightarrow{\pi_{2s,n-2s}} \PP(V_n) $$
	where $\pi''_{2s}((q,f))=(q,f^2)$. At the level of Chow rings, the first morphism coincides with
	$$ A^*_{\GLt}(\PP(V_s))\otimes_{A^*_{\GLt}(\cl{S})} A^*_{\GLt}(\PP(V_{n-2s})) \xrightarrow{\pi''_{2s*} \otimes {\rm id}_*} A^*_{\GLt}(\PP(V_{2s}))\otimes_{A^*_{\GLt}(\cl{S})} A^*_{\GLt}(\PP(V_{n-2s})) $$
	From the previous case we deduce that $$\pi'_{2s*}(h_s^i\cdot h_{n-2s}^j)=2\pi_{2s,n-2s*}( \pi''_{2s*}h^{i+1}_s\cdot h_{n-2s}^j) $$ which concludes the proof of the lemma.
\end{proof}
\begin{rmk}\label{rmk:not ext}
	We cannot extend the previous Lemma to the classes in which $h_s^{2s}$ appears. Indeed, for $s=n/2$, write $\pi'_{n*}h_{n/2}^{n}$ as a polynomial $\alpha_0 h_n^{2n}+\alpha_1 h_n^{2n-1}+...+\alpha_{2n}$. Then we have $$\pi'_{n*}h_{n/2}^{n}\cdot h_n=(\alpha_0 h_n^{2n}+\alpha_1 h_n^{2n}+...\alpha_{2n})\cdot h_n=\alpha_0 (h_n^{n+1}-p_n(h_n))+...+\alpha_{2n}h_n$$ and the first coefficient will always be $2$-divisible, no matter if $\alpha_0$ is even or not.
\end{rmk}
\begin{proof}[Proof of lemma \ref{lm:not highest degree classes are in I}]
	Consider again the commutative diagram
	$$ \xymatrix {
		A^*_{\GLt}(\PP(V_s)\times_{\cl{S}}\PP(V_{n-2s})) \ar[r] \ar[d]^*[@]{\cong} & A^*_{\GLt}(\PP(V_n)) \ar[d]^*[@]{\cong} \\
		A^*_{\PGLt}(\PP(1,2s)\times\PP(1,2n-4s)) \ar[r] \ar[d] & A^*_{\PGLt}(\PP(1,2n)) \ar[d] \\
		A^*_{\SLt}(\PP(1,2s)\times\PP(1,2n-4s)) \ar[r]  & A^*_{\SLt}(\PP(1,2n)) }$$
	where the three horizontal arrows are respectively the pushforward along the morphisms $\pi'_{2s}$, $\pi_{2s}$ and $\pi^{\SLt}_{2s}$. 
	
	Recall that the kernel of the two last vertical maps is $(c_3)$ and that, from \cite{EdiFul}, we already know that ${\rm im}(\pi^{\SLt}_{2s*})$ is contained in ${\rm im}(\pi^{\SLt}_{1*})$. This implies that there exists a cycle $\xi$ such that $\pi'_{2s*}(h_s^i\cdot h_{n-2s}^j) + c_3\cdot\xi$ is contained in ${\rm im}(\pi'_{1*})$. 
	
	Observe that this last ideal is contained in $(2)$ and, by lemma \ref{lm:2-divisibility}, so is $\pi'_{2s*}(h_s^i\cdot h_{n-2s}^j)$. From this we deduce that $c_3\xi$ is contained in $(c_3)$ and $(2)$, but $(c_3)\cap (2)=(0)$, thus $c_3\cdot\xi=0$ and consequently $\pi'_{2s*}(h_s^i\cdot h_{n-2s}^j)$ is contained in ${\rm im}(\pi'_{1*})$.
\end{proof}
So far we have proved that the ideal ${\rm im}(i_*)$ is equal to the sum of the ideal ${\rm im}(\pi'_{1*})$ and the ideal generated by the elements $\pi'_{2s*}(h_s^{2s}\cdot h_{n-2s}^j)$ for $s=1,...,n/2$ and $j=0,...,2n-4s$. We are in position to prove the main result of this subsection.
\begin{proof}[Proof of proposition \ref{pr:generators first reduction}]
	The key observation is that $\pi^{'*}_{2s}h_n=2h_s+h_{n-2s}$. This implies the following chain of equalities:
	\begin{align*}
	\pi'_{2s*}(h_s^{2s}\cdot h_{n-2s}^j)&=\pi'_{2s*}(h_s^{2s}\cdot h_{n-2s}^{j-1}\cdot(h_{n-2s}+2h_s-2h_s))\\
	&=\pi'_{2s*}(h_s^{2s}\cdot h_{n-2s}^{j-1}\cdot\pi^{'*}_{2s}h_n)-2\pi'_{2s*}(h_s^{2s+1}\cdot h_{n-2s}^{j-1})\\
	&=\pi'_{2s*}(h_s^{2s}\cdot h_{n-2s}^{j-1})\cdot h_n - 2\pi'_{2s*}((\alpha_0h_s^{2s}+...+\alpha_{2s})\cdot h_{n-2s}^{j-1}) 
	\end{align*}
	By lemma \ref{lm:not highest degree classes are in I} we see that the cycle $\pi'_{2s*}(h_s^{2s}\cdot h_{n-2s}^j)$ is in the ideal ${\rm im}(\pi'_{1*})+(\pi'_{2s*}(h_s^{2s}\cdot h_{n-2s}^{j-1}))$. Iterating this argument, we see that for every $j>0$ the cycle $\pi'_{2s*}(h_s^{2s}\cdot h_{n-2s}^j)$ is contained in the ideal ${\rm im}(\pi'_{1*})+(\pi'_{2s*}(h_s^{2s}\cdot 1))$. Applying this to every $s$ we conclude the proof of the lemma.
\end{proof}
\subsection{Interlude: computations in the $T$-equivariant setting}
Let $T\subset\GLt$ be the maximal subtorus of diagonal matrices.
In this subsection the fact that we work with the $T$-equivariant Chow ring will be essential. For the sake of clarity, the morphisms between $T$-equivariant Chow rings will be denoted with a $T$ in the apex. Moreover, we will keep using $\ZZ_{(2)}$-coefficients, so that every ring and ideal is assumed to be tensored with $\ZZ_{(2)}$, where not explicitly written. What we have found in the last subsection implies that
$${\rm im}(i_*^T)={\rm im}(\pi^{'T}_{1*})+(\pi^{'T}_{2s*}(h_s^{2s}\cdot 1),\pi^{'T}_{n*}(h_{n/2}^{n})),\quad s=1,...,n/2-1$$
Recall that we defined in the third section the $T$-invariant subvarieties $W_{n,r,l}$ of $\PP(V_n)$ whose points are the pairs $(q,X_0^rX_1^lf)$ (see definition \ref{def:Wnrl}).
\begin{prop}\label{pr:generators im i Z2 coefficients T-equivariant}
	We have ${\rm im}(i_*^T)=(2h_n^2-2n(n-1)c_2,4(n-2)h_n,[W_{n,2,0}])$ inside $A^*_T(\PP(V_n))$.
\end{prop} 
All what we need in order to prove the proposition above is the following lemma:
\begin{lm}\label{lm:generators second reduction}
	We have ${\rm im}(i_*^T)={\rm im}(\pi^{'T}_{1*})+([W_{n,2,2s-2}])_{s=1,...,n/2}$.
\end{lm}
\begin{proof}
	From lemma \ref{lm:Z_classes} we know that the cycle class $[W_{s,1,s-1}]$ contained in the $T$-equivariant Chow ring of $\PP(V_s)$ is a monic polynomial in $h_s$ of degree $2s$. We already know from lemma \ref{lm:not highest degree classes are in I} that the cycles $\pi^{'T}_{2s*}(h_s^{i}\cdot 1)$, for $i<2s$, are in ${\rm im}(\pi^{'T}_{1*})$.
	
	Combining this with our initial observation, we get
	$$ {\rm im}(\pi^{'T}_{1*})+\pi'_{2s*}(h_s^{2s}\cdot 1)\subset {\rm im}(\pi^{'T}_{1*})+(\pi^{'T}_{2s*}([W_{s,1,s-1}]\times 1)) $$
	because we have
	$$\pi^{'T}_{2s*}(h_s^{2s}\cdot 1)=\pi^{'T}_{2s*}([W_{s,1,s-1}]\times 1)-\sum\xi_i\pi^{'T}_{2s*}(h_s^{i}\cdot 1)$$
	for $i<2s$. The other inclusion is obvious because the ideal on the right is by construction contained in ${\rm im}(i^T_*)$, that coincides with the ideal on the left. Thus we actually proved that we have an equality. To finish the proof of the lemma, is enough to observe that, by lemma \ref{lm:pushforward_W_classes}, we have $\pi^{'T}_{2s*}([W_{s,1,s-1}]\times 1)=[W_{n,2,2s-2}]$.
\end{proof}
\begin{proof}[Proof of proposition \ref{pr:generators im i Z2 coefficients T-equivariant}]
	From lemma \ref{lm:Z_classes_intersection} we see that the ideal $([W_{n,2,2s-2}])$, where $s=1,...,n/2$, is actually generated by $[W_{n,2,0}]$.
	
	This plus lemma \ref{lm:generators second reduction} implies that ${\rm im}(i_*^{T})={\rm im}(\pi^{'T}_{1*})+([W_{n,2,0}])$. The fact that the inclusion ${\rm im}(\pi{'T}_{1*})\subset {\rm im}(i^T_*)$ is strict follows from the fact that ${\rm im}(\pi^{'T}_{1*})\subset (2)$ whereas $[W_{n,2,0}]$ is not $2$-divisible, as lemma \ref{lm:Z_classes} shows.
\end{proof}
\subsection{Computations with $\ZZ_{(2)}$-coefficients, part two}
We want to deduce from proposition \ref{pr:generators im i Z2 coefficients T-equivariant} what are the generators of ${\rm im}(i_*)\otimes_{\ZZ}\ZZ_{(2)}$. Again, all the ideals and the Chow rings will be assumed to be tensorized over $\ZZ$ with $\ZZ_{(2)}$, where not explicitly stated. Observe that the ideal ${\rm im}(i^T_*)$ is equal to the sum of ideals ${\rm im}(\pi^{'T}_{r*})$, where $r$ ranges from $1$ to $n$. In particular, proposition \ref{pr:generators first reduction} remains true also in the $T$-equivariant setting, so that we actually know that
$$ {\rm im}(i_*^T)={\rm im}(\pi_{1*}^{' T})+(\pi^{'T}_{2s*}(h_s^{2s}\cdot 1))_{s=1,...,n/2}$$
and from lemma \ref{lm:not highest degree classes are in I}, which also stays true in the $T$-equivariant setting, we know that the cycles of the form $\pi_{2s*}^{'T}(h_s^i\cdot h_{n-2s}^j)$ are in ${\rm im}(\pi^{'T}_{1*})$ for $i=0,...,2s-1$ and $j=0,...,4n-2s$.

We proved in proposition \ref{pr:generators im i Z2 coefficients T-equivariant} that ${\rm im}(i_*^T)$ is equal to ${\rm im}(\pi_{1*}^{' T})$ plus the ideal generated by the cycle class $[W_{n,2,0}]=\pi^{'T}_{2*}([W_{1,1,0}]\times 1)$. The equality
$$[W_{1,1,0}]=(h_1-\lambda_2)(h_1-\lambda_3)=h_1^2-(\lambda_2+\lambda_3)h_1+\lambda_2\lambda_3$$
is immediate to check, using the fact that $\PP(V_1)=\PP(2,1)\times\cl{S}$. Putting all together, we readily deduce that
$$ {\rm im}(i^T_*)={\rm im}(\pi^{'T}_{1*})+([W_{n,2,0}])\subset{\rm im}(\pi^{'T}_{1*})+(\pi^{'T}_{2*}(h_1^2\cdot 1))\subset {\rm im}(i^T_*) $$
which implies the following result:
\begin{cor}
	We have $ {\rm im}(i^T_*)={\rm im}(\pi^{'T}_{1*})+(\pi^{'T}_{2*}(h_1^2\cdot 1))$
\end{cor}
Let us recall the following result (\cite[Lemma $2.1$]{FulVis}), which enables us to pass from the $T$-equivariant setting to the $\GLt$-equivariant one:
\begin{prop}\label{pr:from T to G}
	Let $G$ be a special algebraic group and let $T\subset G$ be a maximal torus.
	\begin{enumerate}[label=(\roman*)]
		\item Let $X$ be a smooth $G$-space and let $I\subset A_G(X)$ be an ideal, then
		$$ IA_T(X)\cap A_G(X)=I $$ 
		\item Let $\{x_1,...,x_r\}$ be a set of variables and let $I\subset A_G[x_1,...,x_r]$ be an ideal, then
		$$ IA_T[x_1,...,x_r]\cap A_G[x_1,...,x_r]=I $$ 
	\end{enumerate}
\end{prop}
Then we see that the ideal ${\rm im}(i_*)$ inside $A^*_{\GLt}(\PP(V_n))$ is equal to the symmetric elements of the image of $i_*^T$ inside the $T$-equivariant Chow ring of $\PP(V_n)$. 
The corollary above gives explicit generators for the ideal $ {\rm im}(i^T_*)$ that are also symmetric in the $\lambda_i$. From this we immediately deduce proposition \ref{pr:generators im i with Z-two coefficients} stated at the beginning of the section.
\section{The Chow ring of $\cl{H}_g$: end of the computation}\label{sec:chow 3}
\noindent
In this section we finish the computation of $A^*(\cl{H}_g)$. Recall that in corollary \ref{cor:generators im i} we proved that
$${\rm im}(i_*)=(2h_n^2-2n(n-1)c_2,4(n-2)h_n,\pi'_{2*}(h_1^2\times [\PP(V_{n-2})]))$$
In order to obtain the relations inside the Chow ring of $\cl{H}_g$, we need to pull back the generators of the ideal above along the $\Gm$-torsor $p:V_n\setminus \sigma_0 \to \PP(V_n)$, where $\sigma_0$ denotes the image of the zero section $\cl{S}\to V_n$. 

We have already computed some of these relations in the third section (see corollary \ref{cor:chow_ring_Hg_intermediate}). Let us call $I$ the ideal appearing in that corollary, that is
$$I=(4(2g+1)\tau,8\tau^2-2g(g+1)c_2,2c_3)$$
By construction, it coincides with the pullback of the ideal ${\rm im}(\pi'_{1*})$. Unfortunately, the ideal ${\rm im}(i_*)$ is not equal to ${\rm im}(\pi'_{1*})$ inside $A^*_{\GLt}(\PP(V_n))$, so we can't conclude that $I$ is the whole ideal of relations, though this claim may still be true, because when pulling back the generator of ${\rm im}(i_*)$ not in ${\rm im}(\pi'_{1*})$ we may obtain a cycle contained in $I$. 

The fact that we do not know an explicit expression for the last generator, namely $\pi'_{2*}(h_1^2\times [\PP(V_{n-2})])$, prevents us from finishing the computation in a direct way.

Recall that in the previous section we deduced that
$${\rm im}(i^T_*)={\rm im}(\pi^{'T}_{1*})+([W_{n,2,0}]) $$
inside the $T$-equivariant Chow ring of $\PP(V_n)$, and observe that the relations inside the Chow ring of $\cl{H}_g$ are exactly the pullback of these symmetric elements along the $\Gm$-torsor $p:(V_n\setminus \sigma_0) \to \PP(V_n)$. 

If $\xi$ is in ${\rm im}(i_*)$, seeing it as an element of the $T$-equivariant Chow ring through the embedding $A^*_{\GLt}(\PP(V_n))\hookrightarrow A^*_T(\PP(V_n))$, then we have $\xi=\alpha\cdot \pi'_{1*}\xi + \beta\cdot [W_{n,2,0}]$. To proceed with our discussion, we need the following technical result:
\begin{lm}\label{lm:criterion for the pullback to be in I}
	Let $\xi=\alpha_0h_n^{2n}+\alpha_1h_n^{2n-1}+...+\alpha_{2n}$ be a cycle in ${\rm im}(i_*)$, considered as an ideal in the $\GLt$-equivariant Chow ring of $\PP(V_n)$, and suppose that $\alpha_{2n}$ is $2$-divisible. Then $p^*\xi$ is in $I$. 
\end{lm}
The proof of the lemma is postponed to the end of the section. Write $\xi$ as a polynomial in $h_n$ of degree less or equal to $2n$. If we prove that $\xi$, written in this form and evaluated in $h_n=0$, is $2$-divisible, then by lemma \ref{lm:criterion for the pullback to be in I} we can conclude that $p^*\xi$ must be in $I$, thus $I$ is the whole ideal of relations.

We already know that every element in the image of $\pi'_{1*}$ is $2$-divisible, so that we only need to check that $\beta\cdot [W_{n,2,0}]$, seen as a polynomial in $h_n$ of degree less than or equal to $2n$ and evaluated in $h_n=0$, is $2$-divisible.

Clearly, it is enough to prove this claim when $\beta=h_n^d$, where $d=0,...,2n$. For matters of clarity, let us work with $\ZZ/2\ZZ$-coefficients, so that what we need to prove is that $\beta\cdot [W_{n,2,0}]$, seen as a polynomial in $h_n$ of degree less or equal to $2n$ and evaluated in $h_n=0$, is equal to $0$. 

Write $[W_{n,2,0}]$ as $h_n^4+\omega_1\cdot h_n^3+...+\omega_4$.  Observe that, with these coefficients, we have $h_n^{2n+1}=0$. This implies that the product $h_n^d\cdot [W_{n,2,0}]$ is equal to
$$ \chi_{4}h_n^{d+4}+\chi_3 \omega_1 h_n^{d+3}+...+\chi_0 \omega_4 h_n^d $$
where $\chi_j=0$ for $j+d>2n$, and equal to $1$ otherwise. In particular, for $d>0$ the evaluation of this polynomial in $h_n=0$ is zero. In other terms, we have shown that $h_n^d\cdot[W_{n,2,0}]$, written as a polynomial in $h_n$ of degree less or equal to $2n$ and evaluated in $h_n=0$, is $2$-divisible for $d>0$. 

Now we only need to prove that $[W_{n,2,0}]$ itself has this property.
Recall from lemma \ref{lm:Z_classes} that
$$ [W_{n,2,0}]=\prod (h_n-\underline{k}\cdot\underline{\lambda})+2\lambda_3\xi \text{ for }\underline{k}\text{ s.t. } |\underline{k}|=n\text{, }k_2<2\text{, }k_0<2$$
The $2$-divisibility of $[W_{n,2,0}]$ when evaluated in $h_n=0$ is then equivalent to studying the $2$-divisibility of the product $\prod (-\underline{k}\cdot\underline{\lambda})$, where $|\underline{k}|=n\text{, }k_2<2\text{, }k_0<2$.

Observe that there are only four triples $\underline{k}$ that verify the conditions above, namely $(0,n,0)$, $(0,n-1,1)$, $(1,n-1,0)$ and $(1,n-2,1)$. This implies that the product above is a multiple of $n\lambda_2$, thus it is $2$-divisible because $n=g+1$ is even. We now give a proof of the technical lemma.
\begin{proof}[Proof of lemma \ref{lm:criterion for the pullback to be in I}]
	We can assume that $\xi$ is not in ${\rm im}(\pi'_{1*})$, otherwise the conclusion is obvious. Moreover, we can also assume that $\xi$ is in ${\rm im}(\pi'_{2s*})$, because we have proved in the last section that ${\rm im}(\pi'_{2s+1*})$ is contained in ${\rm im}(\pi'_{1*})$. Consider again the commutative diagram
	$$\xymatrix {
		A^*_{\GLt}(\PP(V_s)\times_{\cl{S}}\PP(V_{n-2s})) \ar[r] \ar[d]^*[@]{\cong} & A^*_{\GLt}(\PP(V_n)) \ar[d]^*[@]{\cong} \\
		A^*_{\PGLt}(\PP(1,2s)\times\PP(1,2n-4s)) \ar[r] \ar[d] & A^*_{\PGLt}(\PP(1,2n)) \ar[d] \\
		A^*_{\SLt}(\PP(1,2s)\times\PP(1,2n-4s)) \ar[r]  & A^*_{\SLt}(\PP(1,2n)) }$$
	Then it must be true that $\xi=\pi'_{1*}\zeta + c_3\cdot\eta$: indeed we know from \cite{EdiFul} that the image of the last horizontal map is contained in the image of $\pi^{\SLt}_{1*}$, and that the kernel of the last two vertical maps, which are surjective, is generated as an ideal by $c_3$.
	
	Let us briefly comment on the surjectivity of these maps: first recall that $\PP(1,2m)$ is the projectivization of a $\PGLt$-representation, hence $A^*_{\PGLt}(\PP(1,2m))$ is generated by the hyperplane section and the Chern classes $c_2$ and $c_3$. Similarly, the ring $A^*_{\PGLt}(\PP(1,2m)\times\PP(1,2m'))$ is generated by the hyperplane sections of the two factors plus $c_2$ and $c_3$. 
	
	Moreover, it is well known that $A^*_{\SLt}$ is generated by $c_2$: therefore $A^*_{\SLt}(\PP(1,2m)\times \PP(1,2m'))$ is generated by $c_2$ and the hyperplane sections, and the surjectivity of the last two vertical maps in the diagram above follows, as all the projective spaces appearing parametrise forms of even degree.
	
	Observe that we can assume that $\eta$, seen as a polynomial in $h_n$,  has only odd coefficients: indeed, if we write $\eta=\eta'+2\eta''$ then
	$$ c_3\cdot\eta=c_3\cdot\eta'+2c_3\cdot\eta''=c_3\cdot\eta' $$
	We deduce then that $\eta$ must be equal to $h_n\cdot\gamma$, because by hypothesis when we evaluate $\xi$ in $h_n=0$ we must obtain something even, and $\eta$ has only odd coefficients. In the end, we have that $\xi=\pi'_{1*}\zeta + h_n\cdot c_3\cdot \gamma$. We now pull back $\xi$ to $V_n\setminus\sigma_0$, which we saw to be equivalent to substituting $h_n$ with $-2\tau$, so we get:
	$$ p^*\xi=p^*\pi'_{1*}\zeta - 2\tau\cdot c_3 \cdot p^*\gamma=p^*\pi'_{1*}\zeta $$
	where in the last equality we used the relation $2c_3=0$. This concludes the proof of the lemma.
\end{proof}
Putting all together, we have finally proved the
\begin{thm}\label{thm:Chow ring of Hg}
	We have $ A^*(\cl{H}_g)=\mathbb{Z}[\tau,c_2,c_3]/(4(2g+1)\tau,8\tau^2-2g(g+1)c_2,2c_3)$.
\end{thm}
We want to give a geometrical interpretation of the generators of $A^*(\cl{H}_g)$. Recall that in order to do the computations of the last three sections we used the isomorphism $[U'/\GLt\times\Gm]$ obtained in theorem \ref{thm:presentation_Hg}, where $U'$ is the open subscheme of $V_{g+1}$ whose points are pairs $(q,f)$ such that $Q$ and $F$ intersect transversely. 

We also showed, at the end of section \ref{sec:new pres}, that the rank $4$ vector bundle over $\cl{H}_g$ associated to the $\GLt\times\Gm$-torsor $U'$ is the vector bundle $\cl{E}\oplus \cl{L}$, where $\cl{L}$ is the line bundle over $\cl{H}_g$ functorially defined as
$$\cl{L}((\pi:C\to S,\iota))=\pi_*\omega_{C/S}^{\otimes \frac{g+1}{2}}\left(\frac{1-g}{2}W\right) $$
and $\cl{E}$ is the rank $3$ vector bundle over $\cl{H}_g$ functorially defined as
$$\cl{E}((\pi:C\to S,\iota))=\pi_*\omega_{C/S}^{-1}\left(W\right)$$
Then by construction the generator $\tau$ coincides with $c_1(\cl{L})$ and $c_2$ and $c_3$ coincide respectively with $c_2(\cl{E})$ and $c_3(\cl{E})$, whereas $c_1(\cl{E})=0$. This analysis agrees with the one made in the last section of \cite{FulViv}.

\end{document}